\documentclass[a4paper, reqno]{amsart}

\numberwithin{equation}{section}

\usepackage{amsfonts, amsmath,  amssymb, color,colortbl, enumerate, esint, fancyvrb, mathrsfs, tabu, tikz}

\usepackage{stackengine}
\newcommand\xrowht[2][0]{\addstackgap[.5\dimexpr#2\relax]{\vphantom{#1}}}

\usetikzlibrary{positioning}
\VerbatimFootnotes

\usepackage[framemethod=TikZ]{mdframed}

\mdfdefinestyle{MyFrame}{%
    linecolor=blue,
    outerlinewidth=1.5pt,
    roundcorner=20pt,
    innertopmargin=5pt,
    innerbottommargin=\baselineskip,
    innerrightmargin=20pt,
    innerleftmargin=20pt,
    backgroundcolor=gray!10!white}

\title{A note on Fourier restriction and nested Polynomial Wolff axioms}

\theoremstyle{plain}
\newtheorem{theorem}{Theorem}[section]
\newtheorem{lemma}[theorem]{Lemma}

\newtheorem*{key estimate}{Key estimate}

\theoremstyle{definition}
\newtheorem{definition}[theorem]{Definition}

\newtheorem{remark}[theorem]{Remark}
\newtheorem*{acknowledgement}{Acknowledgement}

\newcommand{\ta}{\texttt{a}}
\newcommand{\bta}{\mbox{\small$\ta$}}
\newcommand{\sta}{\mbox{\scriptsize$\ta$}}
\newcommand{\tc}{\texttt{c}}
\newcommand{\btc}{\mbox{\small$\tc$}}
\newcommand{\stc}{\mbox{\scriptsize$\tc$}}

\def\inn#1#2{\langle#1,#2\rangle}

\newcommand{\eps}{\varepsilon}

\newcommand{\N}{\mathbb{N}}
\newcommand{\Z}{\mathbb{Z}}
\newcommand{\R}{\mathbb{R}}

\newcommand{\T}{\mathbb{T}}
\newcommand{\C}{\mathbb{C}}
\newcommand{\W}{\mathbb{W}}

\newcommand{\ud}{\mathrm{d}}

\newcommand{\codim}{\mathrm{codim}\,}
\newcommand{\dist}{\mathrm{dist}}
\newcommand{\supp}{\mathrm{supp}\,}

\renewcommand{\O}{\mathcal{O}}
\newcommand{\Sc}{\mathcal{S}}

\newcommand{\bZ}{\mathbf{Z}}

\newcommand{\cT}{\mathcal{T}}

\newcommand{\BL}[1]{\mathrm{BL}_{#1}}

\newcommand{\alg}{\mathrm{alg}}

\author[J. Hickman]{Jonathan Hickman}
\address{School of Mathematics, The University of Edinburgh, Edinburgh, EH9 3JZ, UK}
\email{jonathan.hickman@ed.ac.uk}

\author[J. Zahl]{Joshua Zahl}
\address{Department of Mathematics, University of British Columbia, Vancouver BC, V6T 1Z2, Canada}
\email{jzahl@math.ubc.ca}

\begin{document}

\begin{abstract} This note records an asymptotic improvement on the known $L^p$ range for the Fourier restriction conjecture in high dimensions. This is obtained by combining Guth's polynomial partitioning method with recent geometric results regarding intersections of tubes with nested families of varieties. 
\end{abstract}

\maketitle




\section{Introduction}




\subsection{Main result} For $n \geq 2$ let $B^{n-1}$ denote the unit ball in $\R^{n-1}$ and consider the \textit{Fourier extension} operator $E$ defined by
\begin{equation*}
E f(x) := \int_{B^{n-1}}
f(\omega)\,e^{i(x_1\omega_1 + \cdots + x_{n-1}\omega_{n-1} + x_n|\omega|^2)}\,  \ud\omega
\end{equation*}
for $f \in L^1(B^{n-1})$. The adjoint form of Stein's restriction conjecture \cite{Stein1979} asserts that the  estimate
\begin{equation}\label{extension}\tag{$\mathrm{R}_p^*$}
\|Ef\,\|_{L^p(\R^n)} \le C_{n,p}\|f\|_{L^p(\R^n)}
\end{equation}
holds for all $p>\frac{2n}{n-1}$. This conjecture is known to hold for $n=2$ due to Fefferman--Stein~\cite{Fefferman1970} but remains open in all other dimensions. The problem has a rich history and is related to a wide range of important questions in harmonic analysis, geometric measure theory, PDE and number theory: see, for instance, the surveys \cite{Wolff1999, Tao2004, Bennett2014} for further details.

The current best partial results on the restriction conjecture are based on the polynomial method \cite{Wang, Guth2018, HR2019}, which was introduced in this context in a seminal work of Guth~\cite{Guth2018}. The purpose of this note is to combine Guth's polynomial method approach with a recent Kakeya-type geometric result from \cite{HRZ,Zahl} in order improve the known bounds on the problem in the high-dimensional regime. 

The attendant numerology is somewhat complicated and consequently it is convenient to express the results in an asymptotic form. In particular, if the restriction conjecture were true, then  \eqref{extension} would hold for
\begin{equation*}
  p>2+2n^{-1}+O(n^{-2}).   
\end{equation*}
Thus, one may consider the $\lambda\ge 2$ for which \eqref{extension} can be confirmed in the range
\begin{equation}\label{asymptotic range}
 p>2+\lambda n^{-1}+O(n^{-2}).
\end{equation}

\begin{theorem}\label{asymptotic thm} \eqref{extension} holds in the range \eqref{asymptotic range} with  $\lambda=2.596...$.
\end{theorem}

The precise form of $\lambda$, involving the irrational real root of a cubic equation, is described in the appendix. 

A comparison of the numerology of Theorem \ref{asymptotic thm} with previous partial results on the restriction conjecture is presented in Figure~\ref{asymptotic table}. 
It is remarked that in certain low dimensions the methods of this paper do not improve upon known results and, in particular, stronger estimates are known in the $n=3$ due to Wang \cite{Wang}. For the current best bounds in low dimensions see Figure~\ref{exponent table}.

\renewcommand{\arraystretch}{1.2}
 \begin{figure}
\begin{tabular}{ ||c|c|| } 
 \hline
  $\lambda=$ &  \\ 
   \hline 
 4 & Tomas \cite{Tomas1975} \\ 
   3 & Bourgain--Guth \cite{BG2011} \\
 8/3 & Guth \cite{Guth2018} \\ 
 2.604... & Hickman--Rogers \cite{HR2019} \\  
 2.596... & Theorem \ref{asymptotic thm} \\
 \hline
\end{tabular}
    \caption{Comparison with previous asymptotic results on the restriction conjecture. 
    }
    \label{asymptotic table}
\end{figure}




\subsection{Polynomial Wolff axioms} Underlying the restriction conjecture are deep geometric problems concerning the continuum incidence theory of long, thin tubes in $\R^n$. To be more precise, given a large parameter $R \geq 1$ define an \emph{$R \times R^{1/2}$-tube} to be a cylinder $T\subset \mathbb{R}^n$ of height $R$ and radius $R^{1/2}$ with arbitrary position and arbitrary orientation. The $R$ versus $R^{1/2}$ scaling arises naturally in the analysis of $Ef$ owing to the quadratic nature of the phase function. The \textit{direction} of an $R \times R^{1/2}$-tube $T$ is defined to be the direction of its coaxial line, which is denoted here by $\mathrm{dir}(T)\in S^{n-1}$. A family $\mathbf{T}$ of $R \times R^{1/2}$-tubes is \emph{direction-separated} if $\{\mathrm{dir}(T) : T \in \mathbf{T}\}$ forms an $R^{-1/2}$-separated subset of the unit sphere.

Effective analysis of $Ef$ relies on understanding of the incidence geometry of direction-separated tube families $\mathbf{T}$. This is formalised by the well-known and celebrated fact that the restriction conjecture implies the \textit{Kakeya conjecture}, where the latter may be loosely interpreted as a bound on the number of possible incidences between tubes in $\mathbf{T}$.

In relation to the incidence theory, a critical case occurs when the tubes from $\mathbf{T}$ tend to align around neighbourhoods of low dimensional, low degree algebraic varieties. The significance of this situation was first highlighted in the context of the restriction conjecture by Guth \cite{Guth2016}, and in the context of the Kakeya conjecture by Guth and the second author \cite{GK2015}. Thus, it is important to understand possible interactions between tubes and varieties, which are typically constrained by the dimension and degree of the variety. A fundamental tool in this direction is the following theorem which, in the language of \cite{GK2015}, states that direction-separated families of tubes satisfy the \textit{polynomial Wolff axioms} (see \cite{GK2015} for a discussion of this terminology).

\begin{theorem}[Polynomial Wolff axioms \cite{KR2018}] \label{poly Wolff axiom} For all~$n > j \ge 1$, $d\ge 1$ and $\varepsilon>0$, there is a constant $C_{n,d,\varepsilon}>0$ such that
\begin{equation*}\#\big\{ T \in \mathbf{T} :  |T \cap B_{r} \cap N_{R^{1/2}} \bZ|\geq r|T| \big\} \leq C_{n,d,\varepsilon} r^{-j}R^{(n + j -1)/2+\varepsilon}
\end{equation*}
whenever $1 \leq R^{1/2} \leq r \le R$, $\mathbf{T}$ is a direction-separated family of $R \times R^{1/2}$-tubes and $\bZ\subset \mathbb{R}^n$ is an algebraic variety of codimension $j$ and degree at most $d$.
\end{theorem}

See also \cite{Guth2016, Zahl2018} for earlier partial results. Here $N_{r}E$ denotes the $r$-neighbourhood of $E$ for any $r > 0$ and $E \subseteq \R^n$ and $B_{r}$ is a choice of ball in $\R^n$ of radius $r$. The relevant algebraic definitions are recalled in \S\ref{alg def sec} below. 

Using the $n=3$ case of Theorem~\ref{poly Wolff axiom}, Guth~\cite{Guth2016} was able to improve the then best bound on the restriction conjecture in $\R^3$. This argument was extended to higher dimensions by Rogers and the first author \cite{HR2019}, combining the ideas from \cite{Guth2016} with those of Guth's study of the higher dimensional problem in \cite{Guth2018}. This led to the previous best known asymptotic for the restriction conjecture (see Figure~\ref{asymptotic table}). 

In \cite{HRZ} and \cite{Zahl} a non-trivial extension of Theorem~\ref{poly Wolff axiom} was concurrently and independently established which, rather than controlling interactions between direction-separated tubes lying close to a single variety, controls interactions between direction-separated tubes and \textit{nested families of varieties}.  

\begin{theorem}[Nested polynomial Wolff axioms \cite{HRZ, Zahl}]\label{nested Wolff thm}
For all $n > m \geq 1$, $d \geq 1$ and $\varepsilon>0$,  there is a constant $C_{n,d,\varepsilon}>0$ such that
\begin{equation*}
\#\bigcap_{j = 1}^{m} \Big\{ T \in \mathbf{T} :  |T \cap B_{r_j} \cap N_{R^{1/2}} \bZ_j|\ge r_j|T|\Big\} \leq C_{n,d,\varepsilon} \Big(\prod_{j = 1 }^{m} r_j^{-1}\Big)R^{(n + m - 1)/2+\varepsilon}
\end{equation*}
holds whenever:
\begin{itemize}
    \item $ 1 \leq R^{1/2} \leq r_j \leq R$ for $1 \leq j \leq m$ and $B_{r_{m}}\subseteq \ldots \subseteq B_{r_1}\subset \R^n$;
    \item $\mathbf{T}$ is a direction-separated family of $R \times R^{1/2}$-tubes;
    \item Each $\bZ_j\subset \mathbb{R}^n$ is an algebraic variety of codimension $j$ and degree at most~$d$.
\end{itemize}
\end{theorem}

In \cite{HRZ, Zahl} Theorem~\ref{nested Wolff thm} was applied to give new bounds on the Kakeya conjecture in the high dimensional regime. In this note the same ideas are transferred into the context of the restriction problem. In particular, by adapting the arguments used to study the restriction problem from \cite{Guth2016, Guth2018, HR2019}, one is led to consider tube interactions with nested families of varieties. There are some differences between the geometric setup in the Kakeya and restriction problems, however, and consequently Theorem~\ref{nested Wolff thm} is not used in the forthcoming analysis \textit{per se}, but rather a variant which is better adapted to the restriction problem. This variant is stated in terms of lines rather than tubes and follows directly from \cite[Lemma 2.11]{Zahl}: see Theorem~\ref{nested line Wolff thm} below. 

{
\renewcommand{\arraystretch}{1.2}
 \begin{figure}
\begin{tabular}{ ||c|c|c||c|c|c|| } 
 \hline
  $n=$ & $p >$  &  & $n=$ & $p > $  & \textbf{} \\ 
   \hline 
 \cellcolor{white} 2 & 4 & Fefferman--Stein \cite{Fefferman1970}  &  \cellcolor{yellow!50} 11 &  \cellcolor{yellow!50} $2+\frac{12597}{49670}$ &  \cellcolor{yellow!50} Theorem \ref{broad thm}\\ 
  \cellcolor{white} 3 & $3+\frac{3}{13}$ & Wang \cite{Wang}  & \cellcolor{white} 12 & $2+\frac{4}{17}$ &  Guth \cite{Guth2018}  \\
  4 &  $2+\frac{1407}{1759}$ & Hickman--Rogers \cite{HR2019}\footnotemark[1] &  \cellcolor{yellow!50}  13 &  \cellcolor{yellow!50} $2+\frac{185725}{878068}$ &  \cellcolor{yellow!50} Theorem \ref{broad thm}  \\  
  \cellcolor{yellow!50} 5 &  \cellcolor{yellow!50} $2+\frac{63}{100}$ &  \cellcolor{yellow!50} Theorem \ref{broad thm}  & \cellcolor{yellow!50} 14 & \cellcolor{yellow!50}  $2+\frac{1671525}{8414731}$ &  \cellcolor{yellow!50} Theorem \ref{broad thm}  \\ 
\cellcolor{white} 6 & $2+\frac{1}{2}$ & Guth \cite{Guth2018}  &  \cellcolor{yellow!50} 15 &  \cellcolor{yellow!50} $2+\frac{2}{11}$ &  \cellcolor{yellow!50} Theorem \ref{broad thm}  \\ 
  \cellcolor{yellow!50} 7 &  \cellcolor{yellow!50} $2+\frac{429}{1018}$ &  \cellcolor{yellow!50} Theorem \ref{broad thm}  &  \cellcolor{yellow!50} 16 &  \cellcolor{yellow!50} $2+\frac{20036013}{116580449}$ &  \cellcolor{yellow!50} Theorem \ref{broad thm}   \\ 
 \cellcolor{white} 8 & $2+\frac{4}{11}$ &  Guth \cite{Guth2018}  & \cellcolor{yellow!50} 17 & \cellcolor{yellow!50} $2+\frac{4}{25}$& \cellcolor{yellow!50} Theorem \ref{broad thm}  \\ 
  \cellcolor{yellow!50} 9 &  \cellcolor{yellow!50} $2+\frac{7293}{23032}$ &  \cellcolor{yellow!50} Theorem \ref{broad thm}  &  \cellcolor{yellow!50} 18 &  \cellcolor{yellow!50} $2+\frac{123751845}{817128103}$ &  \cellcolor{yellow!50} Theorem \ref{broad thm}   \\ 
 \cellcolor{white} 10 & $2+\frac{2}{7}$ &  Guth \cite{Guth2018}  &   \cellcolor{yellow!50}  19 &  \cellcolor{yellow!50} $2+\frac{1}{7}$ &  \cellcolor{yellow!50} Theorem \ref{broad thm} \\
 \hline
\end{tabular}
    \caption{The current state-of-the-art for the restriction problem in low dimensions. New results are \colorbox{yellow!50}{highlighted} and are deduced by combining Theorem \ref{broad thm} with the linear to $k$-broad reduction from \cite{BG2011, Guth2018}.$^2$
    }
    \label{exponent table}
\end{figure}
\footnotetext[1]{See also \cite{Demeter}.}
\footnotetext[2]{These computations were carried out using the following Maple \cite{Maple} code:
 \begin{verbatim}
n := [insert dimension];
p_broad := 2+6/(2*(n-1)+(k-1)*4^(n-k)*(factorial(n-1)/factorial(k-1))^2
*factorial(2*k-1)/factorial(2*n-1)):   p_limit :=2+ 4/(2*n-k):
p_seq := [seq(max(eval(p_broad, k = i), eval(p_limit, k = i)), i = 2 .. n)]: 
new_exponent := min(p_seq);
\end{verbatim}}
\setcounter{footnote}{2}}




\subsection{$k$-broad estimates}

Rather than attempt to prove \eqref{extension} directly, a number of standard reductions are applied to reduce matters to a simpler class of estimates. 

By a now standard \emph{$\varepsilon$-removal argument} (see \cite{Tao1999}) and factorisation theory (see  \cite{Bourgain1991} or ~\cite[Lemma 1]{Carbery1992}), the inequality \eqref{extension} holds for all $p$ in an open range if and only if for all $\varepsilon>0$ and all $R \geq 1$ the local estimates
\begin{equation}\label{local extension}\tag{$\mathrm{R}^*_{p,\mathrm{loc}}$}
\|Eg\|_{L^p(B_{R})}\leq C_{n,p,\varepsilon} R^{\varepsilon}\|g\|_{L^\infty(\R^{n-1})}
\end{equation}
hold in the same range. Here $B_{R}$ denotes an arbitrary ball of radius $R$ in $\R^n$.

Using the Bourgain--Guth method \cite{BG2011, Guth2018}, one may further reduce the problem to working with weaker \textit{$k$-broad estimates}, which take the form
\begin{equation}\label{broad}\tag{$\BL{k}^p$}    
\|Eg\|_{\BL{k}^{p}(B_{R})} \leq C_{n,p,\varepsilon} R^{\varepsilon} \|g\|_{L^\infty(\R^{n-1})}.
\end{equation}
The reader is referred to \cite{Guth2018} for the definition and basic properties of the \textit{$k$-broad norm} appearing on the left-hand side of this inequality.

The main result of this article is the following theorem.

\begin{theorem}\label{broad thm} Let $2 \leq k \leq n-1$ and  
\begin{align}\label{with}
 p\ge p_n(k):= 2+\frac{6}{2(n-1)+(k-1)\prod_{i=k}^{n-1}\frac{2i}{2i+1}}.
\end{align}
Then \eqref{broad} holds for all $\varepsilon>0$ and $R \gg 1$. 
\end{theorem}

Theorem~\ref{asymptotic thm} follows from Theorem~\ref{broad thm} in view of the aforementioned Bourgain--Guth method \cite{BG2011, Guth2018}. In particular, it follows from \cite{BG2011, Guth2018} that \eqref{broad} implies \eqref{local extension} whenever $n\ge 3$ and 
\begin{equation}\label{BG constraints}
2+\frac{4}{2n-k} \le p\le 2+\frac{2}{k-2}.
\end{equation}
Thus, to obtain the best possible estimate \eqref{extension} from Theorem~\ref{broad thm}, one wishes to choose a value of $k$ which optimises the range of $p$ in \eqref{with} subject to the constraints in \eqref{BG constraints}. For a given dimension it is a straightforward exercise to compute the optimal choice of $k$ and the resulting range of linear restriction estimates in low dimensions are tabulated in Figure~\ref{exponent table}. However, in general the optimisation procedure does not produce a compact formula for the explicit $p$ range, hence the convenience of the asymptotic formulation in Theorem~\ref{asymptotic range}. The derivation of the value $\lambda = 2.596...$ for the asymptotic is described in the appendix.

\subsection{Structure of the article} This article is \textbf{not} self-contained and refers heavily back to the work of Guth \cite{Guth2018}, and the reformulation of Guth's induction-on-scale argument as a recursive algorithm in \cite{HR2019}. In \S\ref{notation sec} certain notational conventions are set up; \S\ref{prelim sec} describes various preliminaries including the wave packet decomposition and basic algebraic definitions; 
\S\ref{overview sec} provides an overview of a modified form of the polynomial structural decomposition described in \cite{Guth2018, HR2019} and explains how this modification can be combined with Theorem \ref{nested Wolff thm} to give Theorem~\ref{broad thm}; \S\ref{L2 orthogonality sec} deals with basic orthogonality results used to establish the decomposition in \S\ref{overview sec} whilst the decomposition itself is proven in \S\ref{relating scales sec}, using arguments from \cite{Guth2018, HR2019}; appended is a discussion of the numerology of Theorem~\ref{asymptotic thm}. 

\begin{acknowledgement} The authors would like to thank Keith M. Rogers for numerous helpful discussions. 
\end{acknowledgement}




\section{Notational conventions}\label{notation sec}

In the arguments that follow, the parameters $n$, $k$, $\varepsilon$ are fixed and satisfy the hypotheses of Theorem~\ref{broad thm}. In particular, all implicit constants are allowed to depend on $n$, $k$, and $\eps$.

Our arguments will involve a number of additional admissible parameters
\begin{equation}\label{small parameters}
    \varepsilon^{C} \leq \delta \ll_{\varepsilon} \delta_0 \ll_{\varepsilon} \delta_1 \ll_{\varepsilon} \dots \ll_{\varepsilon} \delta_{n-k} \ll_{\varepsilon} \varepsilon_{\circ} \ll_{\varepsilon} \varepsilon.
\end{equation}
Here $C$ is some dimensional constant and the notation $A \ll_{\varepsilon} B$ for $A, B \geq 0$ indicates that $A \leq \mathbf{C}_{n,\varepsilon}^{-1} B$ for some fixed large admissible constant $\mathbf{C}_{n,\varepsilon} \geq 1$ chosen to satisfy the requirements of the following arguments. 

Given $A, B \geq 0$ and a (possibly empty) list of objects $L$, the notation $A \lesssim_L B$, $A = O_L(B)$ or $B \gtrsim_L A$ indicates that $A \leq C_{n, L} B$ for some constant $C_{n,L}  > 0$ depending only on $n$ and the objects in $L$, whilst $A \sim_L B$ denotes that $A \lesssim_L B$ and $B \lesssim_L A$. Given a large parameter $r \geq 1$, the notation $\mathrm{RapDec}(r)$ is used to denote a non-negative term which is rapidly decreasing in $r$ in the sense that
\begin{equation*}
    \mathrm{RapDec}(r) \lesssim_{\varepsilon,N} r^{-N} \qquad \textrm{for all $N \in \N$.}
\end{equation*}
Such terms frequently appear as `errors' in the arguments. 



\section{Preliminaries}\label{prelim sec}




\subsection{Wave packet decomposition}\label{wave packet sec} For $r \geq 1$ let $\Theta[r]$ denote the set of all balls $\theta$ of radius $r^{-1/2}$ in $\R^{n-1}$ with centres $\omega_{\theta}$ lying in $c_n r^{-1/2}\Z^{n-1} \cap B^{n-1}$ for $c_n := 2^{-1}(n-1)^{-1/2}$. Fix $\psi \in C^{\infty}_c(\R^{n-1})$ with $\supp \psi \subseteq [-c_n, c_n]^{n-1}$ satisfying 
\begin{equation*}
\sum_{k \in \Z^{n-1}} \psi(\,\cdot\, - c_nk) \equiv 1;
\end{equation*}
such a function may be constructed using the Poisson summation formula. In addition, let $\tilde{\psi} \in C^{\infty}_c(\R^{n-1})$ satisfy $\supp \tilde{\psi} \subseteq B^{n-1}$ and $\tilde{\psi}(\omega) = 1$ whenever $\omega \in \supp \psi$. For $\theta \in \Theta[r]$ with $\omega_{\theta} = c_n r^{-1/2}k_{\theta}$ define $\psi_{\theta}(\omega) := \psi(r^{1/2}\omega - c_n k_{\theta})$ and $\tilde{\psi}_{\theta}(\omega) := \tilde{\psi}(r^{1/2}\omega - c_n k_{\theta})$ so that both $\psi_{\theta}$ and $\tilde{\psi}_{\theta}$ are supported in $\theta$ and $\tilde{\psi}_{\theta}(\omega) = 1$ whenever $\omega \in \supp \psi_{\theta}$.

 Writing $\T[r] := \Theta[r] \times r^{1/2}\Z^{n-1}$, the inversion formula for Fourier series allows one to decompose
\begin{equation}\label{wave packet decomposition}
f = \sum_{(\theta,v) \in \T[r]} f_{\theta,v}
\end{equation}
where\footnote{Here $\hat{g}$ denotes the Fourier transform of $g \in L^1(\R^d)$; that is: $\hat{g}(\xi) := \int_{\R^d} e^{-i \inn{x}{\xi}} g(x)\,\ud x$.}
\begin{equation*}
f_{\theta,v}(\omega) := \Big(\frac{r^{1/2}}{2\pi}\Big)^{n-1}e^{i \langle v, \omega \rangle } (f\cdot \psi_{\theta})^\wedge(v)\tilde{\psi}_{\theta}(\omega).
\end{equation*}
The sum \eqref{wave packet decomposition} is referred to as the \emph{wave packet decomposition of~$f$ at scale $r$}. The pairs $(\theta,v) \in \T[r]$ and functions $f_{\theta,v}$ will both be referred to as \emph{(scale $r$) wave packets}. 

The key properties of this decomposition are as follows:




\begin{itemize}
    \item \textbf{Orthogonality between the wave packets.} Combining spatial orthogonality with the Plancherel identity for Fourier series,
\begin{equation*}
    \max_{\theta_* \in \Theta[\rho]}\Big\|\sum_{(\theta,v) \in \mathbb{W}} f_{\theta,v}\Big\|_{L^2(\theta_*)}^2 \sim \max_{\theta_* \in \Theta[\rho]} \sum_{(\theta,v) \in \mathbb{W}} \|f_{\theta,v}\|_{L^2(\theta_*)}^2
\end{equation*}
for any collection of wave packets $\mathbb{W} \subseteq \T[r]$ and $1 \leq \rho \leq r$.
\item \textbf{Spatial concentration.} Given $0 < \delta \ll 1$ as in \eqref{small parameters} and any wave packet $(\theta,v) \in \T[r]$, define the tube
\begin{equation*}
  T_{\theta,v} :=  \big\{ x \in B(0,r)\, :\, |x'+2x_n\omega_{\theta} + v| \leq r^{1/2+\delta}\big\}.
\end{equation*}
 By a simple stationary phase argument (see, for instance, \cite{Tao2003}) the function $Ef_{\theta,v}$ is concentrated on $T_{\theta,v}$ in the sense that
\begin{equation*}
    |Ef_{\theta,v}(x)\chi_{B(0,r) \setminus T_{\theta,v}}(x)| = \mathrm{RapDec}(r)\|f\|_{2} \qquad \textrm{for all $x \in \R^n$.}
\end{equation*}
\end{itemize}
Each $T_{\theta,v}$ is an $R \times R^{1/2}$-tube with coaxial line passing through the point $(-v,0)^{\top} \in \R^n$ in the direction $G(\omega_{\theta}) := (-2\omega_{\theta}, 1)^{\top}$.

\begin{definition} Given $\W \subseteq \T[r]$, a function $f \in L^1(B^{n-1})$ is said to be concentrated on wave packets from $\mathbb{W}$ if 
\begin{equation*}
    \big\|\sum_{(\theta,v)\notin \mathbb{W}} f_{\theta,v}\big\|_{\infty} = \mathrm{RapDec}(r)\|f\|_2.
\end{equation*} 
Here the $\mathrm{RapDec}(r)$ notation is as defined in \S\ref{notation sec}.
\end{definition}




\subsection{Algebraic/geometric definitions}\label{alg def sec} Given real polynomials $P_1, \dots, P_m \in \R[X_1, \dots, X_n]$ let
\begin{equation*}
    Z(P_1, \dots, P_m) := \big\{ z \in \R^n : P_j(z) = 0 \textrm{ for $1 \leq j \leq m$} \} 
\end{equation*}
denote their common zero set. For the purposes of this article, such a set $\bZ$ is referred to as a \textit{variety}. 


Suppose $\bZ := Z(P_1, \dots, P_m)$ is a variety satisfying the additional condition
\begin{equation*}
    \bigwedge_{j=1}^{m} \nabla P_j(z) \neq 0 \qquad \textrm{for all $z \in \bZ$.}
\end{equation*}
In this case, $\bZ$ is said to be a \textit{transverse complete intersection}. Such variety $\bZ$ is a smooth submanifold of $\R^n$ of codimension $m$ and, in particular, has a tangent plane $T_z\bZ$ at every point $z \in \bZ$. It is remarked that the entire Euclidean space $\R^n$ is trivially considered a transverse complete intersection. 

The proof of Theorem~\ref{broad thm} relies on analysing geometric interactions between the tubes $T_{\theta,v}$ arising from the wave packet decomposition and varieties $\bZ$. Given a wave packet $(\theta,v) \in \T[r]$ and $y \in \R^n$ let $T_{\theta,v}(y) := y + T_{\theta,v}$.

\begin{definition} Let $\bZ$ be a codimension $m$ transverse complete intersection and fix a ball $B(y,r) \subseteq \R^n$. The tube $T_{\theta,v}(y)$ associated to a wave packet $(\theta,v) \in \T[r]$ is said to be $r^{-1/2+\delta_m}$-\textit{tangent} to $\bZ$ in $B(y,r)$ if the following conditions hold:
\begin{enumerate}[i)]
    \item $T_{\theta,v}(y) \subseteq N_{r^{1/2 + \delta_m}}\bZ \cap B(y,r)$;
    \item For any $x \in T_{\theta,v}(y)$ and $z \in \bZ \cap B(y,r)$ with $|x - z| \lesssim r^{1/2+\delta_m}$ one has
\begin{equation}\label{tangent def 1}
   \angle\big(G(\omega_{\theta}), T_z\bZ\big) \lesssim r^{-1/2 + \delta_m}.
\end{equation}
\end{enumerate}
Here $G(\omega_{\theta})$ is the direction of the tube $T_{\theta,v}$, as defined at the end of \S\ref{wave packet sec}, and the left-hand side of \eqref{tangent def 1} denotes the (unsigned) angle between this vector and $T_z\bZ$. The $\delta_m$ exponent is as described in \S\ref{notation sec}.
\end{definition}

Such notions of tangency may be expressed more succinctly by introducing the concept of a `grain', similar to that used in \cite{GZ,Zahl}

\begin{definition} For the purposes of this article, a \textit{grain} is defined to be a pair $(S, B_r)$ where $S \subseteq \R^n$ is a transverse complete intersection and $B_r \subset \R^n$ is a ball of some radius $r > 0$. 
The \textit{(co)dimension} of a grain $(S, B_r)$ is the (co)dimension of the variety $S$, whilst its \textit{degree} is the degree of $S$ and its \textit{scale} is the value of the radial parameter $r$.
\end{definition}

\begin{definition}\label{tangent f def} Let $\big(S,B(y,r)\big)$ be a codimension $m$ grain. A function $f \in L^1(B^{n-1})$ is said to be \textit{tangent} to $\big(S,B(y,r)\big)$ if it is concentrated on scale $r$ wave packets belong to the collection
\begin{equation}\label{tangent f 1}
    \big\{ (\theta,v) \in \T[r] : \textrm{  $T_{\theta,v}(y)$ is $r^{-1/2+\delta_m}$-tangent to $S$ in $B(y,r)$} \big\}. 
\end{equation}
\end{definition}

The polynomial Wolff axioms may be used to study the geometry of tubes $T_{\theta,v}(y)$ satisfying the condition in \eqref{tangent f 1}; in particular, Theorem~\ref{poly Wolff axiom} was used in this way to prove restriction estimates in \cite{Guth2016, HR2019}. Here more complex nested geometric structures are considered, which are described by the following definition (see also \cite{HRZ,Zahl}). 

\begin{definition}[Multigrain] A \textit{multigrain} is an $(m+1)$-tuple of grains 
\begin{equation*}
    \vec{S}_{m} = (\mathcal{G}_0, \dots, \mathcal{G}_{m}), \qquad \mathcal{G}_i = (S_i,B_{r_i}) \textrm{ for $0 \leq i \leq m\leq n$}
\end{equation*}
satisfying
\begin{itemize}
    \item $\codim S_i = i$ for $0 \leq i \leq m$,
    \item $S_m\subset S_{m-1}\subset\cdots\subset S_0$,
    \item $B_{r_{m}} \subseteq B_{r_{m-1}} \subseteq \cdots \subseteq B_{r_0}$.
    \end{itemize}
    The parameter $m$ is referred to as the \textit{level} of the multigrain. The \textit{complexity} of the multigrain is defined to be the maximum of the degrees $\deg S_i$ over all $0 \leq i \leq m$. Finally, the \textit{multiscale} of $\vec{S}_{m}$ is the tuple $\vec{r} = (r_0, r_1, \dots, r_{m})$.
\end{definition}

Given multigrains $\vec{S}_{\ell}$ and $\vec{S}_m$ of levels $\ell$ and $m$, respectively, write $\vec{S}_m \preceq \vec{S}_{\ell}$ if $\ell \leq m$ and the grains forming the first $\ell+1$ components of $\vec{S}_m$ agree those of $\vec{S}_{\ell}$.



\subsection{Nested tubes and the nested polynomial Wolff axioms} The proof of Theorem~\ref{broad thm} relies on an incidence estimate for families of tubes which have a certain multi-scale structure.  

\begin{definition}\label{nested tube condition} Let $\vec{S}_{m} = (\mathcal{G}_0, \dots, \mathcal{G}_{m})$ be a multigrain with 
\begin{equation*}
    \mathcal{G}_i = \big(S_i, B(y_i,r_i)\big) \quad \textrm{for $0 \leq i \leq m$.}
\end{equation*}
Define $\T[\vec{S}_{m}]$ to be the set of scale $R := r_0$ wave packets $(\theta_0,v_0) \in \T[R]$ that satisfy:\medskip

\noindent \textbf{Nested tube hypothesis.} There exists $(\theta_i,v_i) \in \T[r_i]$ for $1 \leq i \leq m$ such that
    \begin{enumerate}[i)]
    \item\label{nestedTubeItem1} $\dist(\theta_i,\theta_j) \lesssim r_j^{-1/2}$,
    \item\label{nestedTubeItem2} $\mathrm{dist}\big(T_{\theta_j,v_j}(y_j),\, T_{\theta_i,v_i}(y_i) \cap B(y_j,r_j) \big) \lesssim r_i^{1/2 + \delta}$,
    \item\label{nestedTubeItem3} $T_{\theta_j,v_j}(y_j) \subset N_{r_j^{1/2+\delta_j}}S_j$
\end{enumerate}
hold for all  $0\leq i \leq j \leq m$. In each case $\mathrm{dist}$ can be taken to be the Hausdorff distance.
\end{definition}

The direction set associated to  $\T[\vec{S}_{m}]$ is given by
\begin{equation*}
    \Theta[\vec{S}_{m}] := \big\{ \theta \in \Theta[R] : (\theta,v) \in \mathbb{T}[\vec{S}_{m}] \textrm{ for some $v \in R^{1/2}\Z^{n-1}$} \big\}.
\end{equation*}
Trivially, $\# \Theta[\vec{S}_{m}] \lesssim R^{(n-1)/2}$. However, the nested tube hypothesis further constrains the number of possible directions of the tubes in $\T[\vec{S}_{m}]$.

\begin{lemma}\label{mod Wolff lem}
Let $\vec{S}_{m}$ be a level $m$ multigrain with multiscale $\vec{r}_{m} = (r_0,\dots,r_{m})$ and complexity at most $d$. If $R := r_0$ and the constants in \eqref{small parameters} are chosen appropriately, then
\begin{equation*}
\# \Theta[\vec{S}_{m}]  \lesssim_{n,d}  \Big(\prod_{j = 1}^{m} r_j^{-1/2}\Big) R^{(n-1)/2+\varepsilon_{\circ}}.
\end{equation*}
\end{lemma}

Lemma~\ref{mod Wolff lem} is a direct consequence of the following variant of Theorem~\ref{nested Wolff thm}, which is deduced by combining \cite[Lemma 2.11]{Zahl} with Wongkew's theorem \cite{Wongkew1993}. 

\begin{theorem}[Nested polynomial Wolff axioms \cite{HRZ, Zahl}]\label{nested line Wolff thm}
For all $n > m \geq 1$, $d \geq 1$ and $\varepsilon>0$,  the bound
\begin{equation*}
\#\bigcap_{j = 1}^{m} \Big\{ L \in \mathbf{L} :  \mathcal{H}^1\big(L \cap B_{r_j} \cap N_{\rho_j} \bZ_j\big)\ge r_j\Big\} \lesssim_{n,d,\varepsilon} \Big(\prod_{j = 1 }^{m} \frac{\rho_j}{r_j}\Big)R^{(n-1)/2+\varepsilon}
\end{equation*}
holds whenever:
\begin{itemize}
    \item $(\rho_j)_{j=1}^{m}$ and $(r_j)_{j=1}^{m}$ are non-increasing sequences lying in the interval $[1, R]$ and $R^{-1/2} \leq \rho_1/r_1$;    
     \item The balls $B_{r_j}$ are nested: $B_{r_{m}}\subseteq \ldots \subseteq B_{r_1}\subset \R^{n}$;
    \item $\mathbf{L}$ is a set of lines pointing in $R^{-1/2}$-separated directions;
    \item  Each $\bZ_j\subset \mathbb{R}^n$ is an algebraic variety of codimension $j$ and degree at most~$d$, and the varieties are nested: $\bZ_m\subset \bZ_{m-1}\subset\cdots\subset \bZ_1$.
\end{itemize}
\end{theorem}

The advantage of Theorem~\ref{nested line Wolff thm} compared with Theorem~\ref{nested Wolff thm} is that the former allows for additional flexibility in the choice of the widths  $\rho_j$ of the neighbourhoods of the $\bZ_j$.

\begin{proof}[Proof (of Lemma~\ref{mod Wolff lem})]
Let $\mathbb{T}\subset \mathbb{T}[\vec{S}_{m}]$ be a set of wave packets pointing in different directions with $\#\mathbb{T} = \# \Theta[\vec{S}_{m}]$. For each wave packet $(\theta_0,v_0)\in\mathbb{T}$, let $(\theta_i,v_i)\in \mathbb{T}[r_i]$ for $i =1,\ldots,m$ be the wave packets described in Definition \ref{nested tube condition}. Let $L_{\theta_0,v_0}$ be the line parallel to $T_{\theta_0,v_0}$ that passes through the midpoint of the wave packet $T_{\theta_m,v_m}(y_m)$. It suffices to bound the cardinality of the family of lines $\mathbf{L} := \{L_{\theta_0,v_0}\colon (\theta_0,v_0)\in\mathbb{T}\}$. By construction, (after pigeonholing) one may assume that the lines in $\mathbf{L}$ point in $R^{-1/2}$-separated directions. 

Let $L\in\mathbf{L}$ and let $(\theta_i,v_i)$, $i=1,\ldots,m$, be the corresponding wave packets. By item \ref{nestedTubeItem1}) and \ref{nestedTubeItem2}) from Definition \ref{nested tube condition}, for each index $j=0,\ldots,m$ it follows that
\begin{equation*}
L \cap B(y_j, r_j) \subset N_{Cr_j^{1/2+\delta}} T_{\theta_j,v_j}(y_j).
\end{equation*}
Since $\delta<\delta_j$, by item \ref{nestedTubeItem3}),
\begin{equation*}
\mathcal{H}^1\big(L \cap N_{Cr_j^{1/2+\delta_j}}S_j \cap B(y_j, r_j)\big) \geq r_j. 
\end{equation*}
Applying Theorem \ref{nested line Wolff thm}, one concludes that
\begin{equation*}
\begin{split}
\#\mathbf{L}&\lesssim_{n,d} \Big(\prod_{j = 1 }^{m} \frac{r_j^{1/2+\delta_j}}{r_j}\Big)R^{(n-1)/2+\varepsilon_{\circ}/2}\\
&\leq \Big(\prod_{j = 1}^{m} r_j^{-1/2}\Big) R^{(n-1)/2+\varepsilon_{\circ}/2 + \delta_0+\ldots+\delta_m}.
\end{split}
\end{equation*}
The result now follows, provided $\varepsilon_{\circ}>2(\delta_0+\ldots+\delta_m)$. 
\end{proof}




\section{An overview of the argument}\label{overview sec}



\subsection{Multiscale grains decomposition}\label{multigrain dec sec} The induction-on-scale argument from \cite{Guth2018} may be interpreted as a procedure for decomposing the broad norm $\|Ef\|_{\BL{k,A}^{p}(B_R)}$ into pieces with certain structural properties. Moreover, the relevant structure may be described in terms of multigrains $\vec{S}_{\ell}$ and the tube families $\T[\vec{S}_{\ell}]$, as introduced in \S\ref{prelim sec}. Here a succinct description of this decomposition is provided, based on the algorithms \texttt{[alg 1]} and \texttt{[alg 2]} from \cite{HR2019}.\medskip

Consider a family of Lebesgue exponents $p_i$ for $0 \leq i \leq n-k$ satisfying 
\begin{equation*}
p_{n-k} \geq p_{n-k+1} \geq \dots \geq p_0 =: p \geq 2.
\end{equation*}
and define $0 \leq \alpha_i, \beta_i \leq 1$ in terms of the $p_i$ by
\begin{equation*}
    \alpha_{i} := \Big(\frac{1}{2} - \frac{1}{p_i}\Big)\Big(\frac{1}{2} - \frac{1}{p_{i-1}}\Big)^{-1} \quad \textrm{and} \quad \beta_{i} := \Big(\frac{1}{2} - \frac{1}{p_i}\Big)\Big(\frac{1}{2} - \frac{1}{p_0}\Big)^{-1}
\end{equation*}
for $1 \leq i \leq n - k$ and $\alpha_0 :=: \beta_0 := 1$.\medskip

\paragraph{\underline{\texttt{Input}}}
Fix $R \gg 1$ and let  $f \colon B^{n-1} \to \C$ be smooth and bounded and, without loss of generality, assume that $f$ satisfies the \emph{non-degeneracy hypothesis} 
\begin{equation*}
\|Ef\|_{\BL{k,A}^{p}(B_{\>\!\!R})} \geq C_{\textrm{hyp}} R^{\varepsilon}\|f\|_{L^2(B^{n-1})}
\end{equation*}
where $C_{\textrm{hyp}}$ and $A \in \N$ are admissible constants which are chosen sufficiently large so as to satisfy the forthcoming requirements.\medskip

\paragraph{\underline{\texttt{Output}}} The algorithm outputs the following objects: 
\begin{itemize}
  \item $\mathcal{O}$ a finite collection of open subsets of $\R^n$ of diameter at most $R^{\varepsilon_{\circ}}$.
  \item A codimension $0 \leq m \leq n-k$ integer parameter $1 \leq A_{m+1} \leq A$. 
    \item An $(m+1)$-tuple of:
    \begin{itemize}
        \item Scales $\vec{r} = (r_0, \dots, r_m)$ satisfying $R = r_0 > r_1 > \dots > r_m$;
        \item  Large and (in general) non-admissible parameters $\vec{D} = (D_1, \dots, D_{m+1})$.
    \end{itemize}
    \item For $0 \leq \ell \leq m$ a family $\vec{\Sc}_{\ell}$ of level $\ell$ multigrains. Each $\vec{S}_{\ell} \in \vec{\Sc}_{\ell}$ has multiscale $\vec{r}_{\ell} = (r_0, \cdots, r_{\ell})$ and complexity $O_{\varepsilon}(1)$. The families have a nested structure in the sense that for each $1 \leq \ell \leq m$ and each $\vec{S}_{\ell} \in \vec{\Sc}_{\ell}$,  there exists some $\vec{S}_{\ell-1} \in \vec{\Sc}_{\ell-1}$ such that $\vec{S}_{\ell} \preceq \vec{S}_{\ell-1}$.
    \item For $0 \leq \ell \leq m$  an assignment of a function $f_{\vec{S}_{\ell}}$ to each $\vec{S}_{\ell} \in \vec{\Sc}_{\ell}$. Each $f_{\vec{S}_{\ell}}$ is tangent to $(S_{\ell}, B_{r_{\ell}})$, the final component of $\vec{S}_{\ell}$, in the sense of Definition~\ref{tangent f def}.
\end{itemize}

The above data is chosen so that the following properties hold:\\

\paragraph{\underline{Property i)}} The inequality
\begin{equation}\label{P 1}\tag{P-i}
\|Ef\|_{\BL{k,A}^p(B_{R})} \leq  M(\vec{r}, \vec{D}) \|f\|_{L^2(B^{n-1})}^{1-\beta_m} \Big( \sum_{O \in \O} \|Ef_{O}\|_{\BL{k,A_{m+1}}^{p_m}(O)}^{p_m}\Big)^{\frac{\beta_m}{p_m}}
\end{equation}
holds for 
\begin{equation*}
    M(\vec{r}, \vec{D}) := \Big(\prod_{i=1}^{m}D_i\Big)^{m\delta}\Big(\prod_{i=1}^{m} r_i^{(\beta_{i-1}-\beta_i)/2}D_i^{(\beta_{i-1} - \beta_{m})/2}\Big).
\end{equation*}

\paragraph{\underline{Property ii)}} 
\begin{equation}\label{P 2}\tag{P-ii}
 \sum_{O \in \O} \|f_{O}\|_2^{2} \lesssim_{\varepsilon} \Big(\prod_{i = 1}^{m + 1} D_{i}^{1 + \delta}\Big) R^{O(\varepsilon_{\circ})}\|f\|_{L^2(B^{n-1})}^{2}. 
\end{equation}

\paragraph{\underline{Property iii)}}   For $1 \leq \ell \leq m$,
\begin{equation}\label{P 3}\tag{P-iii}
 \max_{O \in \O}\|f_{O}\|_2^2 \lesssim_{\varepsilon} r_\ell^{-\ell/2}\prod_{i= \ell + 1}^{m+1}r_{i}^{-1/2} D_{i}^{-(n-i) + \delta} R^{O(\varepsilon_{\circ})}\max_{\vec{S}_{\ell} \in \vec{\Sc}_{\ell}} \|f_{\vec{S}_{\ell}}\|_2^2,
\end{equation}
where $r_{m+1} := 1$.\medskip

Properties i), ii) and iii) are stated explicitly in \cite{HR2019}: see Remark~\ref{ref HR rmk} below. The present argument requires one further property which does not appear in \cite{HR2019} but nevertheless follows as a consequence of the decomposition procedure described there. In order to state this property, for each multigrain $\vec{S}_{\ell} \in \vec{\mathcal{S}}_{\ell}$ let
\begin{equation*}
   f_{\vec{S}_{\ell}}^{\#} := \sum_{(\theta,v) \in \T[ \vec{S}_{\ell}]} f_{\theta,v}
\end{equation*}
where $\T[ \vec{S}_{\ell}]$ is the collection of scale $R$ wave packets introduced in Definition~\ref{nested tube condition}.\medskip 

\paragraph{\underline{Property iv)}}   For $1 \leq \ell \leq m$, 
\begin{equation}\label{P 4}\tag{P-iv}
    \|f_{\vec{S}_{\ell}}\|_{L^2(B^{n-1})}^2 \lesssim_{\varepsilon} r_{\ell}^{\ell/2}\Big(\prod_{i = 1}^{\ell}r_i^{-1/2}D_i^{\delta} \Big)  R^{O(\varepsilon_{\circ})}\|f_{\vec{S}_{\ell}}^{\#}\|_{2}^2
\end{equation}
holds for all $\vec{S}_{\ell} \in \vec{\mathcal{S}}_{\ell}$.\medskip

To verify \eqref{P 4} it is necessary to unpack some of the details of the polynomial partitioning algorithms described in \cite{Guth2018, HR2019}. This is postponed until \S\ref{relating scales sec} below. Presently the above properties are combined with Lemma~\ref{mod Wolff lem} to conclude the proof of the Theorem~\ref{broad thm}.

\begin{remark}\label{ref HR rmk} The decomposition corresponds to the final output of the algorithm \texttt{[alg 2]} from \cite{HR2019}. In particular:
\begin{itemize}
    \item \eqref{P 1} corresponds to \cite[(55)]{HR2019} and follows from Property I of \texttt{[alg 1]} and Property 1 of \texttt{[alg 2]}.
    \item \eqref{P 2} corresponds to the second displayed equation on p.269 of \cite{HR2019} and follows from Property II of \texttt{[alg 1]} and Property 2 of \texttt{[alg 2]}.
    \item \eqref{P 3} corresponds to  \cite[(57)]{HR2019} and follows from Property III of \texttt{[alg 1]} and Property 3 of \texttt{[alg 2]}.
    \item \eqref{P 4} is related to \cite[(63)]{HR2019}, which follows from the local versions of the estimates in Property III of \texttt{[alg 1]} and Property 3 of \texttt{[alg 2]}. In \S\ref{relating scales sec} below \eqref{P 4} is established by adapting the argument used to prove \cite[(63)]{HR2019}. For this various auxiliary results are required, which are discussed in \S\ref{L2 orthogonality sec}. 
\end{itemize}
Note that the indexing used above is slightly different to that appearing in \cite{HR2019} since here the $\vec{S}_{\ell}$ are indexed according to codimension rather than dimension.
\end{remark}

\begin{remark}
Note that the multiscale grains decomposition detailed above outputs a set of functions $\{f_{{\vec S}_{\ell}}\colon{\vec S}_{\ell}\in \vec{\Sc}_{\ell} \} $ and states certain inequalities that these functions satisfy. This remark provides an informal description of these functions and how they are constructed.

Recall that a multigrain $\vec{S}_{\ell}\in\vec{\Sc}_{\ell}$ is a tuple $(\mathcal{G}_0,\mathcal{G}_1,\ldots,\mathcal{G}_m)$. Here $\mathcal{G}_0$ corresponds to a choice of ball of radius $R$, and for each index $i=1,\ldots,m$, $\mathcal{G}_i$ is a pair $(S_i, B_{r_i})$, where $S_i$ is a variety of codimension $i$ and $B_{r_i} = B(y_i,r_i)$ is a ball of radius $r_i$. 

When $\ell = 0$ and $\vec{S}_{0}=(\mathcal{G}_0)$, then $f_{\vec{S}_{0}}=\sum_{(\theta,v) \in \W_0}f_{\theta,v},$ where $\W_0 \subset \T[r_0]$ consists of those scale $r_0 = R$ wave packets $(\theta,v)$ for which the associated tube $T_{\theta,v}$ intersects $B_{r_0}$.

Now suppose $1\leq i \leq m$ and  $f_{\vec{S}_{i-1}}$ has been defined for all multigrains $\vec{S}_{i-1}$ of level $i-1$. Fixing $\vec{S}_i$ a multigrain of level $i$, the function $f_{\vec{S}_i}$ may be described as follows. Let $\vec{S}_{i-1}$ be the (unique) multigrain from $\vec{\Sc}_{i-1}$ with $\vec{S}_{i} \preceq \vec{S}_{i-1}$. Then heuristically, $f_{\vec{S}_{i}}$ should be thought of as 
\begin{equation}\label{heuristic for fVecS}
f_{\vec{S}_{i}} \textrm{``}=\textrm{''} \sum_{(\theta,v)\in \W_i}(f_{\vec{\mathcal{S}}_{i-1}})_{\theta,v},
\end{equation}
where $\W_i$ consists of those wave packets $(\theta,v)\in \T[r_i]$ for which the associated scale $r_i$ tube $T_{\theta,v}(y_i)$ is $r_i^{-1/2+\delta_i}$ tangent to $S_i$ in $B(y_i,r_i)$. 

Note the quotation marks around the equality in \eqref{heuristic for fVecS}, which are intended warn the reader that \eqref{heuristic for fVecS} is not true in a literal sense. The reason for this is that $f_{\vec{S}_{i}}$ is constructed from $f_{\vec{S}_{i-1}}$ over many steps using an iterative process. In each one of these steps one performs a new wave packet decomposition at some intermediate scale between $r_{i-1}$ and $r_i$, and some of the wave packets might be discarded at each stage. The arguments have been carefully crafted so that these discarded wave packets can be ignored without adversely affecting our estimates, but caution is needed because some of the more straightforward, na\"{\i}ve statements which would follow from \eqref{heuristic for fVecS} are not true. 
\end{remark}

Since each $O \in \O$ has diameter at most $R^{\varepsilon_{\circ}}$, trivially one may bound
\begin{equation*}
    \|Ef_{O}\|_{\BL{k,A_{m-1}}^{p_m}(O)} \lesssim_{\varepsilon} R^{O(\varepsilon_{\circ})}\|f_O\|_{L^2(B^{n-1})}
\end{equation*}
This trivial bound can be applied to the right-hand side of \eqref{P 1} and combined with \eqref{P 2} and the definition of $M(\vec{r}, \vec{D})$ to deduce that
  \begin{equation*}
      \|Ef\|_{\BL{k,A}^{p}(B_{R})} \lesssim_{\varepsilon} \prod_{i = 1}^{m+1} r_{i}^{\frac{\beta_{i-1}-\beta_{i}}{2}}D_{i}^{\frac{\beta_{i-1}}{2} - (\frac{1}{2}-\frac{1}{p})+O(\delta)}R^{O(\varepsilon_{\circ})}\|f\|_2^{2/p}\max_{O \in \O} \|f_O\|_2^{1-2/p}.
\end{equation*}
The problem is now to bound the maximum appearing on the right-hand side of this expression. 

\subsection{Improvement using the multiscale polynomial Wolff axioms} To obtain an improved result, the multigrain structure is exploited using Lemma~\ref{mod Wolff lem}.

\begin{lemma}\label{multi scale lem} 
For $m \leq \ell \leq n$,
\begin{equation*}
   \max_{\vec{S}_\ell \in \vec{\Sc}_\ell} \|f_{\vec{S}_{\ell}}^{\#}\|_2^2 \lesssim_{\varepsilon} \Big(\prod_{i=1}^{\ell}r_{i}^{-1/2}\Big) R^{\varepsilon_{\circ}}\|f\|_{L^{\infty}(B^{n-1})}^2.
\end{equation*}
\end{lemma}

\begin{proof} Letting $\|\,\cdot\,\|_{L^2_{\mathrm{avg}}(\theta)}$ denote the $L^2$-norm taken with respect to the normalised (to have mass 1) Lebesgue measure on $\theta$, one may write
\begin{equation*}
     \|f_{\vec{S}_{\ell}}^{\#}\|_2^2 \sim R^{-(n-1)/2} \sum_{\theta \in \Theta(R)} \|f_{\vec{S}_{\ell}}^{\#}\|_{L^2_{\mathrm{avg}}(\theta)}^2,
\end{equation*}
where the right-hand sum can of course be restricted to those $\theta$ that intersect $\supp f_{\vec{S}_{\ell}}^{\#}$. Since $f_{\vec{S}_{\ell}}^{\#}$ is the sum of $f_{\theta,v}$ over all $(\theta,v) \in \T[\vec{S}_{\ell}]$, it follows that
\begin{equation*}
    \#\big\{\theta \in \Theta(R) : \theta \cap \supp f_{\vec{S}_{\ell}}^{\#} \neq \emptyset \big\} \lesssim \#\Theta[\vec{S}_{\ell}] 
\end{equation*}
where $\Theta[\vec{S}_{\ell}]$ is as defined in Definition~\ref{nested tube condition}. Consequently, 
\begin{equation*}
    \|f_{\vec{S}_{\ell}}^{\#}\|_{L^2(B^{n-1})}^2 \lesssim  R^{-(n-1)/2} \cdot \#\Theta[\vec{S}_{\ell}] \cdot \max_{\theta \in \Theta(R)} \|f_{\vec{S}_{\ell}}^{\#}\|_{L^2_{\mathrm{avg}}(\theta)}^2.
\end{equation*}
By orthogonality between the wave packets,
\begin{equation*}
    \max_{\theta \in \Theta(R)} \|f_{\vec{S}_{\ell}}^{\#}\|_{L^2_{\mathrm{avg}}(\theta)}^2 \lesssim \max_{\theta \in \Theta(R)} \|f\|_{L^2_{\mathrm{avg}}(\theta)}^2 \leq \|f\|_{L^{\infty}(B^{n-1})}^2.
    \end{equation*}
On the other hand, Lemma~\ref{mod Wolff lem} implies that 
\begin{equation*}
    \#\Theta[\vec{S}_{\ell}] \lesssim_{\varepsilon} \big(\prod_{i=1}^{\ell}r_{i}^{-1/2}\big)R^{(n-1)/2 + \varepsilon_{\circ}} .
\end{equation*}
Combining the three previous displays yields the desired result. 
\end{proof}

Combining Lemma~\ref{multi scale lem} with property \eqref{P 4} of the decomposition,
\begin{equation*}
 \max_{\vec{S}_\ell \in \vec{\Sc}_\ell} \|f_{\vec{S}_{\ell}}\|_2^2 \lesssim_{\varepsilon}  r^{\ell/2} \Big( \prod_{i=1}^{\ell}r_{i}^{-1/2}\Big)^2 \Big(\prod_{i=1}^{\ell}D_{i}^{\delta}\Big) R^{O(\varepsilon_{\circ})}\|f\|_{L^{\infty}(B^{n-1})}^2.
\end{equation*}
Substituting this estimate into the right-hand side of \eqref{P 3}, one concludes that
\begin{equation}\label{multi scale est}
    \max_{O \in \O}\|f_{O}\|_2^2 \lesssim_{\varepsilon} \Big(\prod_{i=1}^{m}r_{i}^{-1/2} D_{i}^{\delta}\Big)\Big(\prod_{i=1}^{\ell}r_{i}^{-1/2}\Big)\Big( \prod_{i = \ell+1}^{m+1}D_{i}^{-(n-i)}\Big) R^{O(\varepsilon_{\circ})} \|f\|_{L^{\infty}(B^{n-1})}^2
\end{equation}
for all $0 \leq \ell \leq m$.

\subsection{Fixing the exponents and concluding the argument}\label{fixing exps sec}  Varying the $\ell$ parameter in \eqref{multi scale est} produces $m+1$ different bounds. Combining these inequalities by taking a weighted geometric mean, one arrives at the following key estimate.\bigskip

\begin{mdframed}[style=MyFrame]
\begin{key estimate} Let $0 \leq \gamma_0, \dots, \gamma_m \leq 1$ satisfy $\sum_{j = 0}^{m} \gamma_{j} = 1$. Then
\begin{equation*}
    \max_{O \in \O}\|f_{O}\|_2^2 \lesssim_{\varepsilon} \prod_{i=1}^{m+1}r_{i}^{-\frac{1 + \sigma_{i}}{2}} D_{i}^{-(n-i)(1-\sigma_{i}) + O(\delta)}  R^{O(\varepsilon_{\circ})}\|f\|_{L^{\infty}(B^{n-1})}^2,
\end{equation*}
where $\sigma_{i} := \sum_{j=i}^{m} \gamma_j$ for $0 \leq i \leq m$ and $\sigma_{m+1} := 0$.
\end{key estimate}
\end{mdframed}
\bigskip

The key estimate may be plugged into the earlier inequality 
  \begin{equation*}
    \|Ef\|_{\BL{k,A}^{p}(B_{R})} \lesssim_{\varepsilon} \prod_{i = 1}^{m+1} r_{i}^{\frac{\beta_{i-1}-\beta_{i}}{2}}D_{i}^{\frac{\beta_{i-1}}{2} - (\frac{1}{2}-\frac{1}{p})+O(\delta)}R^{O(\varepsilon_{\circ})}\|f\|_2^{2/p}\max_{O \in \O} \|f_O\|_2^{1-2/p}
\end{equation*}
to yield the bound
  \begin{equation}\label{chosing exp 1}
      \|Ef\|_{\BL{k,A}^{p}(B_{R})} \lesssim_{\varepsilon} \prod_{i = 1}^{m+1} r_{i}^{X_{i}}D_{i}^{Y_{i}+O(\delta)}  R^{O(\varepsilon_{\circ})}\|f\|_{L^{\infty}(B^{n-1})}
\end{equation}
where, recalling $p =p_0$, the $X_i$, $Y_i$ exponents are given by
\begin{align*}
    X_{i}&:= \frac{\beta_{i-1}-\beta_{i}}{2}-\frac{(1 +\sigma_{i})}{2}\Big(\frac{1}{2}-\frac{1}{p_0}\Big); \\
    Y_{i}&:= \frac{\beta_{i-1}}{2} - \big(1+(n-i)(1-\sigma_{i})\big)\Big(\frac{1}{2}-\frac{1}{p_0}\Big).
\end{align*}

At this point the values of the free parameters $p_i$ and $\gamma_i$ are fixed. Define
\begin{align*}
    \gamma_j &:= \frac{n-m-1}{2} \cdot \frac{1}{(n-j)(n-j-1)} \cdot \prod_{i=n-m}^{n-j} \frac{2i}{2i+1} \quad \textrm{for $1 \leq j \leq m$,} \\
\gamma_0 &:= 1 - \sum_{j=1}^{m}\gamma_j,
\end{align*}
so that the $\gamma_j$ sum to 1. One may also show, using some algebra, that $0 \leq \gamma_j \leq 1$. Let $p_m := 2\cdot \frac{n-m}{n-m -1}$ and define the remaining $p_i$ in terms of the $\gamma_j$ via the equation
\begin{equation*} 
\Big(\frac{1}{2} - \frac{1}{p_{i}}\Big)^{-1}= 2n - m - i + \sum_{j=i+1}^{m}(j-i)\gamma_j.
\end{equation*} 
With these parameter choices one may verify using simple (yet rather lengthy) algebraic manipulations that $X_{i}, Y_{i} = 0$ for all $1 \leq i \leq m$ and $Y_{m+1} = 0$ and, furthermore, 
\begin{equation*}
p_0 =2+\frac{6}{2(n-1)+(n-m-1)\prod_{i=n-m}^{n-1}\frac{2i}{2i+1}}.
\end{equation*}
These computations are similar to those appearing in \cite{HR2019} and are left to the interested reader. 
The worst case scenario (in the sense that $p_0$ is maximised) occurs when $m$ is as large as possible. Recall that $0 \leq m \leq n-k$; by taking $m = n-k$ one obtains the exponent $p_n(k)$ featured in Theorem~\ref{broad thm}. 

With the above choice of exponents, the $r_{i}^{X_{i}}D_{i}^{Y_{i}}$ factors in \eqref{chosing exp 1} are admissible (indeed, they are equal to 1). However, the $D_i^{O(\delta)}$ factors may still be large. To deal with this, one may slightly perturb the exponents to decrease the $Y_i$ value so as to ensure that the $Y_i + O(\delta)$ is non-positive. This results in a slightly larger $p$ value and, since $\delta > 0$ is arbitrary, establishes Theorem~\ref{broad thm} in the open range $p > p_n(k)$. The closed range of estimates then follows trivially via H\"older's inequality, using the fact that $R^{\varepsilon}$ losses are permitted in the constants.\hfill $\Box$




\section{$L^2$ orthogonality and transverse equidistribution revisited}\label{L2 orthogonality sec}




\subsection{Comparing wave packets at different scales}\label{comparing scales sec} Fix a large scale $r \geq 1$ and a smaller scale $r^{1/2} \leq \rho \leq r$. Given $g \in L^1(B^{n-1})$ one may form the scale $r$ wave packet decomposition 
\begin{equation}\label{big wave packet}
    g = \sum_{(\theta,v) \in \T[r]} g_{\theta,v} + \mathrm{RapDec}(r)\|g\|_2.
\end{equation}
Alternatively, given a ball $B(y,\rho)$ with $y \in B(0,r)$ one may form the scale $\rho$ wave packet decomposition of $g$ over $B(y,\rho)$. In particular, let $\phi \colon \R^n \times \R^{n-1} \to \R$ denote the phase function associated to the extension operator $E$; that is,
\begin{equation*}
    \phi(x;\omega) := \inn{x'}{\omega} + x_n|\omega|^2, \qquad x = (x',x_n) \in \R^{n-1} \times \R,\,\, \omega \in \R^{n-1}. 
\end{equation*}
Write
\begin{equation}\label{recentre g}
    \tilde{g}(\omega) := e^{i \phi(y;\omega)}g(\omega)
\end{equation}
so that $Eg(x) = E\tilde{g}(\tilde{x})$ for $x = y + \tilde{x}$. One may then decompose
\begin{equation}\label{small wave packet}
    \tilde{g} = \sum_{(\tilde{\theta},\tilde{v}) \in \T[\rho]} \tilde{g}_{\tilde{\theta},\tilde{v}} + \mathrm{RapDec}(r)\|g\|_2. 
\end{equation}

Following the discussion in \cite[\S 7]{Guth2018}, the purpose of this section is to compare properties of the decompositions \eqref{big wave packet} and \eqref{small wave packet} under various hypotheses on $g$. In particular, it is useful to study properties of functions formed by restricted sums of wave packets. 

\begin{definition} For $\widetilde{\W}\subseteq \T[\rho]$ define $\uparrow\! \widetilde{\W}$ to be the set of all $(\theta,v) \in \T[r]$ for which there exists some $(\tilde{\theta},\tilde{v}) \in \widetilde{\W}$ satisfying
    \begin{enumerate}[i)]
     \item $\dist(\tilde{\theta},\theta) \lesssim \rho^{-1/2}$;
    \item $\dist(T_{\tilde{\theta},\tilde{v}}(y),\, T_{\theta,v} \cap B(y,\rho)\big) \lesssim r^{1/2+\delta}$.
\end{enumerate}
 Furthermore, for any $g \in L^1(B^{n-1})$ let
\begin{equation*}
 \tilde{g}|_{\widetilde{\W}} := \sum_{(\tilde{\theta},\tilde{v}) \in \widetilde{\W}} \tilde{g}_{\tilde{\theta},\tilde{v}}, \qquad   g|_{\uparrow\widetilde{\W}} := \sum_{(\theta,v) \in \uparrow\widetilde{\W}} g_{\theta,v}.
\end{equation*}
\end{definition}

The definition of $\uparrow \widetilde{\W}$ is motivated by the following.

\begin{lemma}\label{W concentration lem} Given $g \in L^1(B^{n-1})$ and $\widetilde{\W}\subseteq \T[\rho]$, one has
\begin{equation*}
    \tilde{g}|_{\widetilde{\W}} = \sum_{(\tilde{\theta},\tilde{v}) \in \widetilde{\W}} \big(g|_{\uparrow \widetilde{\W}}\big)\;\widetilde{}\!\!_{\tilde{\theta},\tilde{v}} + \mathrm{RapDec}(r)\|g\|_2. 
\end{equation*}
\end{lemma}

In particular, combining Lemma~\ref{W concentration lem} with the orthogonality property of the wave packets,
\begin{equation}\label{basic W L2}
     \|\tilde{g}|_{\widetilde{\W}}\|_2^2 \lesssim \|g|_{\uparrow \widetilde{\W}}\|_2^2 + \mathrm{RapDec}(r)\|g\|_2^2. 
\end{equation}
The proof of Lemma~\ref{W concentration lem} follows directly from the observations in \cite[\S7]{Guth2018}.

\begin{proof}[Proof (of Lemma~\ref{W concentration lem})] Forming the scale $r$-wave decomposition and using the linearity of the mapping $f \mapsto \tilde{f}_{\tilde{\theta},\tilde{v}}$, one may write
\begin{equation*}
\tilde{g}|_{\widetilde{W}} = \sum_{(\tilde{\theta},\tilde{v}) \in \widetilde{\W}}\sum_{(\theta,v) \in \T[r]} \big(g|_{\theta,v}\big)\;\widetilde{}\!\!_{\tilde{\theta},\tilde{v}} + \mathrm{RapDec}(r)\|g\|_2.
\end{equation*}
An integration-by-parts argument (see, for instance, \cite[Lemma 7.1]{Guth2018}) shows that the function $\big(g_{\theta,v}\big)_{\tilde{\theta},\tilde{v}}$ is rapidly decaying whenever $(\tilde{\theta},\tilde{v})$ fails to satisfy the conditions
\begin{equation}\label{W concentration 1}
   \dist(\theta,\tilde{\theta}) \lesssim \rho^{-1/2} \quad \textrm{and} \quad |v - \tilde{v}| \lesssim r^{1/2+\delta}.
\end{equation}
 On the other hand, if $(\tilde{\theta},\tilde{v})$ does satisfy \eqref{W concentration 1}, then it is not difficult to show that 
\begin{equation*}
    T_{\tilde{\theta},\tilde{v}}(y) \subseteq C \cdot T_{\theta,v} \cap B(y,\rho);
\end{equation*}
indeed, this is essentially part of the content of \cite[Lemma 7.2]{Guth2018}. Combining these observations, one deduces that $\big(g_{\theta,v}\big)_{\tilde{\theta},\tilde{v}}$ is rapidly decaying whenever $(\theta,v) \not\in \uparrow\! \widetilde{\W}$, and the desired identity follows. 
\end{proof}




\subsection{Transverse equidistribution estimates revisited}

Continuing to work with the scales $r^{1/2} \leq \rho \leq r$ from the previous section,  fix a transverse complete intersection $\bZ$ in $\R^n$ of codimension $j$ and degree at most $d$. Suppose $g \in L^1(B^{n-1})$ concentrated on scale $r$ wave packets belonging to the family
\begin{equation*}
    \T_{\bZ} := \big\{(\theta,v) \in \T[r] : \textrm{$T_{\theta,v}$ is $r^{-1/2+\delta_j}$-tangent to $\bZ$ in $B(0,r)$}\big\}. 
\end{equation*}
It is shown in \cite[\S 7]{Guth2018} that $\tilde{g}$ is concentrated on scale $\rho$ wave packets which belong to the union of the sets
\begin{equation}\label{rho tangent}
    \widetilde{\T}_{\bZ+b} := \big\{ (\tilde{\theta},\tilde{v}) \in \T[\rho] : \textrm{$T_{\tilde{\theta},\tilde{v}}(y)$ is $\rho^{-1/2+\delta_j}$-tangent to $\bZ+b$ in $B(y,\rho)$}\big\}
\end{equation}
as $b$ varies over vectors in $\R^n$ satisfying $|b| \lesssim r^{1/2 + \delta_j}$. Thus, one is led to consider the functions
\begin{equation*}
    \tilde{g}_b := \sum_{(\tilde{\theta},\tilde{v}) \in \widetilde{\T}_{\bZ+b}} \tilde{g}_{\tilde{\theta},\tilde{v}}.
\end{equation*}

In this subsection certain $L^2$ bounds for the functions $\tilde{g}_b$ obtained in \cite{Guth2018} are generalised, in view of establishing property \eqref{P 4} of the decomposition from the previous section.

\begin{lemma}\label{trans eq lem} Suppose $\bZ$ is a transverse complete intersection of codimension $j$ and degree at most $d$ and $b \in \R^n$ with $|b| \lesssim r^{1/2 + \delta_j}$. If $g \in L^1(B^{n-1})$ is concentrated on wave packets from $\T_{\bZ}$ and $\widetilde{\W} \subseteq \widetilde{\T}_{\bZ + b}$, then
\begin{equation*}
    \|\tilde{g}|_{\widetilde{\W}}\|_{2}^2 \lesssim r^{O(\delta_j)}(r/\rho)^{-j/2} \|g|_{\uparrow\widetilde{\W}}\|_{2}^2 + \mathrm{RapDec}(r)\|g\|_2^2.
\end{equation*}
\end{lemma}
\begin{remark} By the orthogonality properties of the wave packets, 
\begin{equation*}
    \|g|_{\uparrow\widetilde{\W}}\|_{2}^2  \lesssim \|g\|_2^2. 
\end{equation*}
On the other hand, if $\widetilde{\W} = \widetilde{\T}_{\bZ + b}$, then $\tilde{g}|_{\widetilde{\W}} = \tilde{g}_b$ and so Lemma~\ref{trans eq lem} implies that
\begin{equation*}
    \|\tilde{g}_b\|_2^2 \lesssim r^{O(\delta_j)} (r/\rho)^{-j/2} \|g\|_2^2,
\end{equation*}
which is precisely the estimate from  \cite[Lemma 7.6]{Guth2018}. 
\end{remark}

The proof of Lemma~\ref{trans eq lem} follows from a minor modification of the argument used to establish \cite[Lemma 7.6]{Guth2018}. In particular, the key ingredient is an auxiliary inequality from \cite[Lemma 7.5]{Guth2018}; in order to recall this lemma, a few preliminary definitions are in order. 

Partition $\T[r]$ into disjoint sets $\T_{\kappa,w}$ indexed by $(\kappa,w) \in \cT := \Theta[\rho]\times r^{1/2}\Z^{n-1}$ satisfying 
\begin{equation*}
    \dist(\theta, \kappa) \lesssim \rho^{-1/2} \quad \textrm{and} \quad |v + (\partial_{\omega}\phi)(y;\omega_{\theta}) - w| \lesssim r^{1/2} \quad \textrm{for all $(\theta,v) \in \T_{\kappa,w}$.}
\end{equation*}
Accordingly, write
\begin{equation*}
    g_{\kappa,w} := \sum_{(\theta,v) \in \T_{\kappa,w}} g_{\theta,v} \qquad \textrm{for all $(\kappa,w) \in \cT$}
\end{equation*}
so that
\begin{equation}\label{trans eq 1}
    g = \sum_{(\kappa,w) \in \cT} g_{\kappa,w} + \mathrm{RapDec}(r)\|g\|_2. 
\end{equation}
This decomposition satisfies the following properties:
\begin{itemize}
    \item By the $L^2$-orthogonality between the wave packets, 
    \begin{equation}\label{trans eq 2}
    \|g\|_2^2 \sim \sum_{(\kappa,w) \in \cT} \|g_{\kappa,w}\|_2^2 \qquad \textrm{and} \qquad  \|g_{\kappa,w}\|_2^2 \sim \sum_{(\theta,v) \in \T_{\kappa,w}} \|g_{\theta,v}\|_2^2 . 
\end{equation}
\item Each $(g_{\kappa,w})\;\widetilde{}\;$ is concentrated on scale $\rho$ wave packets belonging to 
\begin{equation*}
  \widetilde{\T}_{\kappa,w} := \big\{(\tilde{\theta},\tilde{v}) \in \T[\rho] : \dist(\tilde{\theta},\kappa) \lesssim \rho^{-1/2} \textrm{ and } |\tilde{v} - w| \lesssim r^{1/2} \big\};  
\end{equation*}
see \cite[Lemma 7.3]{Guth2018}. The sets $\widetilde{\T}_{\kappa,w}$ form a finitely-overlapping cover of $\T[\rho]$ as $(\kappa,w)$ varies over $\cT$.
\item If $(\theta,v) \in \T_{\kappa,w}$ and $(\tilde{\theta},\tilde{v}) \in \widetilde{\T}_{\kappa,w}$ for some $(\kappa,w) \in \cT$, then 
\begin{equation}\label{trans eq 3}
    \dist(\theta,\tilde{\theta}) \lesssim \rho^{-1/2} \quad \textrm{and} \quad \dist\big(T_{\tilde{\theta},\tilde{v}}(y),\, T_{\theta,v} \cap B(y,\rho)\big) \lesssim r^{1/2+\delta}.
\end{equation}
Indeed, by the definition of $\T_{\kappa,w}$ and $\widetilde{\T}_{\kappa,w}$, it follows directly that 
\begin{equation}\label{trans eq 4}
    \dist(\theta,\tilde{\theta}) \lesssim \rho^{-1/2} \quad \textrm{and} \quad |v + (\partial_{\omega}\phi)(y;\omega_{\theta}) - \tilde{v}| \lesssim r^{1/2}.
\end{equation}
This immediately establishes the first inequality in \eqref{trans eq 3} and, in fact, the second inequality in \eqref{trans eq 3} also follows from \eqref{trans eq 4}; see \cite[Lemma 7.2]{Guth2018}. 
\end{itemize}
Furthermore, the following \textit{transverse equidistribution estimate} holds for functions simultaneously concentrated on wave packets from $\T_{\bZ}$ and some $\T_{\kappa, w}$.

\begin{lemma}[{\cite[Lemma 7.5]{Guth2018}}]\label{aux trans eq lem} Suppose $h \in L^1(B^{n-1})$ is concentrated on wave packets from $\T_{\bZ} \cap \T_{\kappa, w}$ for some $(\kappa,w) \in \cT$ and $|b| \lesssim r^{1/2+\delta_j}$. Then,
\begin{equation*}
    \|\tilde{h}_b\|_2^2 \lesssim r^{O(\delta_j)} (r/\rho)^{-j/2} \|h\|_2^2. 
\end{equation*}
\end{lemma}

The auxillary estimate from Lemma~\ref{aux trans eq lem} may be combined with the various properties of the functions $\tilde{g}_{\kappa,w}$ described above in order to establish Lemma~\ref{trans eq lem}. 

\begin{proof}[Proof (of Lemma~\ref{trans eq lem})] Since $g \mapsto \tilde{g}|_{\widetilde{\W}}$ is a linear operation on $L^1(\R^{n-1})$, it follows from \eqref{trans eq 1} that
\begin{equation}\label{trans eq 5}
\tilde{g}|_{\widetilde{\W}} = \sum_{(\kappa, w) \in \cT} (g_{\kappa, w})\;\widetilde{}\;|_{\widetilde{\W}} + \mathrm{RapDec}(r)\|g\|_2.
\end{equation} 
Each $(g_{\kappa, w})\;\widetilde{}\;$ is concentrated on scale $\rho$ wave packets belonging to $\widetilde{\T}_{\kappa, w}$ and, consequently,
\begin{equation}\label{trans eq 6}
    (g_{\kappa, w})\;\widetilde{}\;|_{\widetilde{\W}} = \sum_{(\tilde{\theta},\tilde{v}) \in \widetilde{\T}_{\kappa, w} \cap \widetilde{\W}} (g_{\kappa, w})\;\widetilde{}_{\!\!\tilde{\theta}, \tilde{v}} + \mathrm{RapDec}(r)\|g\|_2. 
\end{equation}
Combining \eqref{trans eq 5} and \eqref{trans eq 6}, one deduced that
\begin{equation*}
    \tilde{g}|_{\widetilde{\W}} = \sum_{\substack{(\kappa, w) \in \cT \\ \widetilde{\T}_{\kappa, w} \cap \widetilde{\W} \neq \emptyset}} \sum_{(\tilde{\theta},\tilde{v}) \in \widetilde{\T}_{\kappa, w} \cap \widetilde{\W}} (g_{\kappa, w})\;\widetilde{}_{\!\!\tilde{\theta}, \tilde{v}} + \mathrm{RapDec}(r)\|g\|_2
\end{equation*}
and thus, since the $\widetilde{\T}_{\kappa, w}$ are finite-overlapping, by the orthogonality between the wave packets,
\begin{equation}\label{trans eq 7}
    \|\tilde{g}|_{\widetilde{\W}}\|_2^2 \lesssim \sum_{\substack{(\kappa, w)  \in \cT \\ \widetilde{\T}_{\kappa, w} \cap \widetilde{\W} \neq \emptyset}} \|(g_{\kappa, w})\;\widetilde{}\;|_{\widetilde{\W}} \|_2^2 + \mathrm{RapDec}(r)\|g\|_2^2.
\end{equation}
As $\widetilde{\W} \subseteq \widetilde{\T}_{\bZ+b}$, again using the orthogonality between the wave packets
\begin{equation}\label{trans eq 8}
    \|(g_{\kappa, w})\;\widetilde{}\;|_{\widetilde{\W}} \|_2^2 \lesssim \|(g_{\kappa, w})\;\widetilde{}_{\!\!b}\,\|_2^2.
\end{equation}
Since $g_{\kappa, w}$ is concentrated on wave packets belonging to $\T_{\bZ} \cap \T_{\kappa, w}$, one may apply Lemma~\ref{aux trans eq lem} to conclude that
\begin{equation}\label{trans eq 9}
    \|(g_{\kappa, w})\;\widetilde{}_{\!\!b}\,\|_2^2 \lesssim r^{O(\delta_j)}(r/\rho)^{-j/2} \|g_{\kappa, w}\|_2^2. 
\end{equation}
Combining \eqref{trans eq 7}, \eqref{trans eq 8} and \eqref{trans eq 9} together with the second orthogonality relation in \eqref{trans eq 2}, one obtains
\begin{equation*}
    \|\tilde{g}|_{\widetilde{\W}}\|_2^2 \lesssim \sum_{\substack{(\kappa, w)  \in \cT \\ \widetilde{\T}_{\kappa, w} \cap \widetilde{\W} \neq \emptyset}} \sum_{(\theta,v) \in \T_{\kappa, w}} \|g_{\theta,v}\|_2^2.
\end{equation*}
Thus, by yet another application of the orthogonality property, the problem is reduced to showing that
\begin{equation}\label{trans eq 10}
    \bigcup_{\substack{(\kappa, w)  \in \cT \\ \widetilde{\T}_{\kappa, w} \cap \widetilde{\W} \neq \emptyset}} \T_{\kappa, w} \subseteq \uparrow \! \widetilde{\W}. 
\end{equation}
This last step follows from \eqref{trans eq 3}. Indeed, suppose $(\theta,v) \in \T[r]$ belongs to the left-hand set in \eqref{trans eq 10} so that there exists some $(\kappa, w) \in \cT$ such that $(\theta,v) \in \T_{\kappa, w}$ and $\widetilde{\T}_{\kappa, w} \cap \widetilde{\W} \neq \emptyset$. If $(\tilde{\theta},\tilde{v}) \in \widetilde{\T}_{\kappa, w} \cap \widetilde{\W}$, then $(\tilde{\theta},\tilde{v}) \in \widetilde{\W}$ satisfies \eqref{trans eq 3}, and therefore $(\theta,v) \in \uparrow \! \widetilde{\W}$ by the definition of the latter set. 
\end{proof}




\subsection{Repeatedly refining the wave packets}\label{repeat refine sec} 

This subsection deals with a technical lemma which is useful when one wishes to repeatedly form refinements of the wave packet decomposition at a given scale. 

Given $\W \subseteq \T[r]$ let $\W^*\subseteq \T[r]$ denote the slightly enlarged set of wave packets
\begin{equation*}
    \W^* := \big\{ (\theta,v) \in \T[r] : \dist(\theta,\tilde{\theta}) \lesssim r^{-1/2} \textrm{ and } |v - \tilde{v}| \lesssim r^{1/2+\delta} \textrm{ for some $(\tilde{\theta},\tilde{v}) \in \W$}\big\}. \end{equation*}

\begin{lemma}\label{repeat refine lem} If $\W_1, \W_2 \subseteq \T[r]$ and $g \in L^1(B^{n-1})$, then
\begin{equation*}
    \big\|\big(g|_{\W_1}\big)|_{\W_2}\big\|_2 \lesssim \big\|g|_{\W_1 \cap \W_2^*}\big\|_2 + \mathrm{RapDec}(r)\|g\|_2.
\end{equation*}
\end{lemma}

\begin{proof} Fix $(\tilde{\theta},\tilde{v}) \in \W_2$ and note that
\begin{equation}\label{repeat lem 1}
    \big(g|_{\W_1}\big)_{\tilde{\theta},\tilde{v}} = \sum_{(\theta,v) \in \W_1} \big(g_{\theta,v}\big)_{\tilde{\theta},\tilde{v}}.
\end{equation}
As in the proof of Lemma~\ref{W concentration lem}, the function $\big(g_{\theta,v}\big)_{\tilde{\theta},\tilde{v}}$ is rapidly decaying whenever $(\tilde{\theta},\tilde{v}) \notin \T_{\theta,v}$ where
\begin{equation*}
    \T_{\theta,v} := \big\{ (\tilde{\theta},\tilde{v}) \in \T[r] : \dist(\theta,\tilde{\theta}) \lesssim r^{-1/2} \textrm{ and } |v - \tilde{v}| \lesssim r^{1/2+\delta}  \big\};
\end{equation*}
see, for instance, \cite[Lemma 7.1]{Guth2018}. The condition $(\tilde{\theta},\tilde{v}) \notin \T_{\theta,v}$ is equivalent to $(\theta,v) \notin \T_{\tilde{\theta},\tilde{v}}$  and so \eqref{repeat lem 1} implies that
\begin{equation*}
 \big(g|_{\W_1}\big)_{\W_2} = \sum_{(\tilde{\theta},\tilde{v}) \in \W_2}   \sum_{(\theta,v) \in \W_1 \cap \T_{\tilde{\theta},\tilde{v}}} \big(g_{\theta,v}\big)_{\tilde{\theta},\tilde{v}} + \mathrm{RapDec}(r)\|g\|_2. 
\end{equation*}
Since $\# \T_{\tilde{\theta},\tilde{v}} = O(1)$, one may apply $L^2$-othogonality together with the Cauchy--Schwarz inequality to deduce that the inequalities
\begin{align*}
  \big\|\big(g|_{\W_1}\big)_{\W_2}\big\|_2^2 \lesssim \sum_{(\tilde{\theta},\tilde{v}) \in \W_2}   \sum_{(\theta,v) \in \W_1 \cap \T_{\tilde{\theta},\tilde{v}}} \big\|\big(g_{\theta,v}\big)_{\tilde{\theta},\tilde{v}}\big\|_2^2 \\
  \lesssim  \sum_{(\theta,v) \in \W_1 \cap \W_2^*} \sum_{(\tilde{\theta},\tilde{v}) \in \T[r]}\big\|\big(g_{\theta,v}\big)_{\tilde{\theta},\tilde{v}}\big\|_2^2 
\end{align*}
hold up to the inclusion of a rapidly decaying error. Further application of $L^2$-orthogonality then yields the desired estimate. 
\end{proof}




\section{Relating the scales: verifying Property iv)}\label{relating scales sec}

\subsection{The first algorithm}  Throughout this section let $p \geq 2$ be fixed and
\begin{equation*}
    \varepsilon^{C} \leq \delta \ll_{\varepsilon} \delta_0 \ll_{\varepsilon} \delta_1 \ll_{\varepsilon} \dots \ll_{\varepsilon} \delta_{n-k} \ll_{\varepsilon} \varepsilon_{\circ} \ll_{\varepsilon} \varepsilon
\end{equation*}
be the family of small parameters described in \S\ref{notation sec}. It will be useful to also work with auxiliary numbers $\tilde{\delta}_j$ defined by
\begin{equation*}
    \big(1 - \tilde{\delta}_{j+1} \big) \big( \tfrac{1}{2} + \delta_{j+1}\big) = \tfrac{1}{2} + \delta_j, 
\end{equation*}
so that $\delta_j/2 \leq \tilde{\delta}_j \leq 2 \delta_j$ for all $0 \leq j \leq n - k$. 
\medskip

\noindent\underline{\texttt{Input}} The algorithm \texttt{[alg 1$^*$]} takes as its input:
\begin{itemize}
    \item A grain $\big(\bZ, B(y,r)\big)$ of codimension $m$.
    \item A function $f \in L^1(B^{n-1})$ which is tangent to $\big(\bZ, B(y,r)\big)$.
    \item An admissible large integer $A \in \N$. 
\end{itemize}

\medskip
%
%
%
\noindent\underline{\texttt{Output}} The $j$th stage of \texttt{[alg 1$^*$]} outputs:
\begin{itemize}
    \item A choice of spatial scale $\rho_j \geq 1$ satisfying $\rho_j \leq \rho_{j-1}/2$ and
\begin{equation*}\label{radius bounds}
\rho_j \leq r^{(1-\tilde{\delta}_{m+1})^{\#_{\sta}(j)}}   \quad \textrm{and} \quad \rho_j \leq \frac{r}{2^{\#_{\stc}(j)}}  
\end{equation*}
  for certain integers $\#_{\bta}(j), \#_{\btc}(j) \in \N_0$ satisfying $\#_{\bta}(j) + \#_{\btc}(j) = j$. 
  \item A family of subsets $\O_j$ of $\R^n$ referred to as \emph{cells}. Each cell $O_j \in \O_j$ is contained in some $\rho_j$-ball $B_{O_j} = B(y_{O_j}, \rho_j)$. 
\item A collection of functions $(f_{O_j})_{O_j \in \O_j}$. For each cell $O_j$ there is a translate $\bZ_{O_j} := \bZ + x_{O_j}$ such that $f_{O_j}$ is tangent to the grain $(\bZ_{O_j},B_{O_j})$. 
  \item A large integer $d \in \N$ which depends only on the admissible parameters and $\deg \bZ$.
  \end{itemize}

Moreover, the components of the ensemble are defined so as to ensure that, for certain coefficients
\begin{equation}\label{alg 1 coeff}
   C^{\mathrm{I}}_{j,\delta}(d,r), \; C^{\mathrm{II}}_{j,\delta}(d), \; C^{\mathrm{III}}_{j,\delta}(d,r), \; C^{\mathrm{IV}}_{j,\delta}(d,r) \lesssim_{d,\delta} r^{\varepsilon_{\circ}} d^{\#_{\stc}(j)\delta}
\end{equation}
and $A_j := 2^{-\#_{\sta}(j)}A \in \N$, the following properties hold:\medskip

\paragraph{\underline{Property I}} Most of the mass of $\|Ef\|_{\BL{k,A}^p(B_{r})}^p$ is concentrated over the $O_j \in \O_j$:
\begin{equation} \tag*{$(\mathrm{I})_j$}
      \|Ef\|_{\BL{k,A}^p(B_{r})}^p \leq C^{\mathrm{I}}_{j,\delta}(d,r) \sum_{O_{j} \in \O_{j}} \|Ef_{O_j}\|_{\BL{k,A_j}^p(O_{j})}^p + j r^{-N} \|f\|_{L^2(B^{n-1})}^p
\end{equation}
for some large fixed $N \in \N$.\\

\paragraph{\underline{Property II}} The functions $f_{O_j}$ satisfy
\begin{equation}\tag*{$(\mathrm{II})_j$}
    \sum_{O_{j} \in \O_{j}} \|f_{O_j}\|_2^2 \leq C^{\mathrm{II}}_{j,\delta}(d)d^{\#_{\stc}(j)}  \|f\|_2^2.
\end{equation}
\paragraph{\underline{Property III}} Each $f_{O_j}$ satisfies
\begin{equation}\tag*{$(\mathrm{III})_j$}
    \|f_{O_{j}}\|_{L^2(B^{n-1})}^2 \leq C_{j,\delta}^{\mathrm{III}}(d, r) \Big(\frac{r}{\rho_{j}}\Big)^{-m/2} d^{-\#_{\stc}(j)(n-m-1)}  \|f\|_{L^2(B^{n-1})}^2.
\end{equation}

Properties I, II and III are stated explicitly in the description of \texttt{[alg 1]} from \cite[\S9]{HR2019}. The modified algorithm \texttt{[alg 1$^*$]} includes an additional property, described presently.

For $\W \subseteq \T[\rho_j]$ let $\uparrow^j \W$ denote the set of wave packets $(\theta,v) \in \T[r]$ satisfying 
\begin{equation*}
    \dist(\theta,\theta_j) \leq c_j \rho_j^{-1/2} \quad \textrm{and} \quad \dist\big(T_{\theta_j,v_j}(y_{O_j}),\, T_{\theta,v}(y) \cap B_{O_j}\big) \leq c_j r^{1/2 + \delta}
\end{equation*}
for some $(\theta_j, v_j) \in \W$. Here $(c_j)_{j=0}^{\infty}$ is positive sequence which is bounded above by an absolute constant $C_{\circ}$ and chosen so as to satisfy the forthcoming requirements of the argument. Furthermore, let
\begin{equation*}
    f_{O_j}|_{\W} := \sum_{(\theta_j,v_j) \in \W} (f_{O_j})|_{\theta_j,v_j} \qquad \textrm{and} \qquad f|_{\uparrow^j \W} := \sum_{\theta,v \in \uparrow^j \W} f_{\theta,v}.
\end{equation*}

 \paragraph{\underline{Property IV}} For any $\W \subseteq \T[\rho_j]$, each $f_{O_j}$ satisfies
\begin{equation}\tag*{$(\mathrm{IV})_j$}
 \|f_{O_{j}}|_{\W}\|_{2}^2 \leq C_{j,\delta}^{\mathrm{IV}}(d, r) \Big(\frac{r}{\rho_{j}}\Big)^{-m/2}\|f|_{\uparrow^j\W}\|_{2}^2.
\end{equation}

This concludes the description of the output of \texttt{[alg 1$^*$]}.\medskip
%
%
\paragraph{\underline{\texttt{Stopping conditions}}} The algorithm has two stopping conditions which are labelled \texttt{[tiny]} and \texttt{[tang$^*$]}. 
\begin{itemize}
\item[\texttt{Stop:[tiny]}] The algorithm terminates if $\rho_j \leq r^{\tilde{\delta}_{m+1}}$.
\end{itemize}
In view of the additional Property IV above, the second stopping condition is slightly modified compared with that of \texttt{[alg 1]} of \cite[\S9]{HR2019}.
\begin{itemize}
\item[\texttt{Stop:[tang$^*$]}]Let $C_{\textrm{\texttt{tang}}}$ and $C_{\alg}$ be large, fixed dimensional constants and $\tilde{\rho} := \rho_j^{1 - \tilde{\delta}_m}$. The algorithm terminates if there exist
\end{itemize}
\begin{itemize}
    \item $\Sc$ a collection of grains $(S,B_{\tilde{\rho}})$ of codimension $m+1$, scale $\tilde{\rho}$ and degree at most $C_{\alg}d$;
    \item An assignment of a function $f_S$ to each\footnote{Here, by an abuse of notation, $S$ is used to denote the grain $(S,B_{\tilde{\rho}})$.} $S \in \Sc$ which is tangent to $(S,B_{\tilde{\rho}})$
\end{itemize}
such that the following inequalities hold:\medskip

\paragraph{\underline{Condition I}}
\begin{equation*}
    \sum_{O_j \in \O_j} \|Ef_{O_j}\|_{\BL{k,A_j}^p(O_{j})}^p \leq C_{\textrm{\texttt{tang}}}  \sum_{S \in \Sc} \|Ef_S\|_{\BL{k,A_j/2}^p(B_{\tilde{\rho}})}^p.
\end{equation*}

\paragraph{\underline{Condition II}}

\begin{equation*}
   \sum_{S \in \Sc}\|f_{S}\|_{L^2(B^{n-1})}^2  \leq C_{\textrm{\texttt{tang}}}r^{n\tilde{\delta}_m}\sum_{O_j \in \O_j}\|f_{O_j}\|_{L^2(B^{n-1})}^2.
\end{equation*}

\paragraph{\underline{Condition III}}
\begin{equation*}
    \max_{S \in \Sc}\|f_{S}\|_{2}^2 \leq C_{\textrm{\texttt{tang}}}\max_{O_{j} \in \O_{j} }\|f_{O_j}\|_{2}^2.
\end{equation*}

Conditions I, II and III are stated explicitly in the description of the stopping condition for \texttt{[alg 1]} from \cite[\S9]{HR2019}. The modified algorithm \texttt{[alg 1$^*$]} includes an additional condition.\medskip

\paragraph{\underline{Condition IV}} Given $(S, B(\tilde{y},\tilde{\rho})) \in \Sc$ there exists some $O_j \in \O_j$ such that 
\begin{equation*}
    \|\tilde{f}_S|_{\W}\|_{2}^2 \leq C_{\textrm{\texttt{tang}}} \|f_{O_j}|_{\uparrow\W}\|_{2}^2
    \end{equation*}
    holds for all $\W \subseteq \T[\tilde{\rho}]$. Here $\uparrow\!\W$ is the set of all $(\theta, v) \in \T[\rho_j]$ for which there exists some $(\tilde{\theta},\tilde{v}) \in \W$ satisfying
\begin{equation*}
    \dist(\tilde{\theta},\theta) \lesssim \tilde{\rho}^{-1/2}, \quad \dist(T_{\tilde{\theta},\tilde{v}}(\tilde{y}),\, T_{\theta,v}(y_{O_j}) \cap B(\tilde{y}, \tilde{\rho})\big) \lesssim \rho_j^{1/2+\delta}
\end{equation*} 
for $y_{O_j}$ the centre of $B_{O_j}$, whilst
\begin{equation*}
\tilde{f}_S|_{\W} := \sum_{(\tilde{\theta}, \tilde{v}) \in \W} (f_S)\;\widetilde{}\!_{\tilde{\theta},\tilde{v}} \qquad \textrm{and} \qquad f_{O_j}|_{\uparrow \W} := \sum_{(\theta, v) \in \uparrow\W} (f_{O_j})_{\theta,v}.
\end{equation*}
The function $\tilde{f}_S$ is as defined in \eqref{recentre g}, taking $y$ to be the centre of $B_{\tilde{\rho}}$. 




\subsection{Ensuring Property IV} The modified algorithm \texttt{[alg 1$^*$]} is obtained by combining the recursive step from \texttt{[alg 1]} from \cite[\S9]{HR2019} with the $L^2$-orthogonality results from \S\ref{L2 orthogonality sec}. The main task is to verify Property IV holds. For this, it is useful to work with an explicit formula for the constant  
\begin{equation}\label{C IV formula}
C_{j,\delta}^{\mathrm{IV}}(d, r) := d^{j\delta} r^{\bar{C}\#_{\bta}(j)\delta_m},
\end{equation}
where $\bar{C} \geq 1$ is a suitably large absolute constant. Note that this agrees with the definition of $C_{j,\delta}^{\mathrm{III}}(d, r)$ used in \cite[\S9]{HR2019} and the stopping condition \texttt{[tiny]} together with \eqref{radius bounds} ensure \eqref{alg 1 coeff} holds in this case. 

The $f_{O_j}$ are defined recursively exactly as in \texttt{[alg 1]}. Recall that there are two cases to consider: the \textit{cellular-dominant} case and \textit{algebraic-dominant} case. For either situation the functions $f_{O_{j+1}}$ are obtained from the $f_{O_j}$ via the same procedure. In particular, supposing $\rho_{j+1}$, $O_{j+1}$, $B_{O_{j+1}}$ and $x_{O_{j+1}}$ have already been defined, the function $f_{O_{j+1}}$ has the following form: for two sets of wave packets $\T[O_{j+1}] \subseteq \T[\rho_j]$ and $\widetilde{\T}[O_{j+1}] \subseteq \T[\rho_{j+1}]$ define
\begin{equation*}
    \tilde{f}_{O_{j+1}} := \tilde{g}_{O_{j+1}} |_{\tilde{\T}[O_{j+1}]} \quad \textrm{for} \quad g_{O_{j+1}} := f_{O_j}|_{\T[O_{j+1}]}.
\end{equation*}
Here, given $g \in L^1(B^{n-1})$, the function $\tilde{g}$ is defined with respect to the ball $B_{O_{j+1}}$ as in \S\ref{comparing scales sec} (that is, taking $y = y_{O_{j+1}}$ in \eqref{recentre g}). The following table shows the choices of $\T[O_{j+1}]$ and $\widetilde{\T}[O_{j+1}]$ used in the cellular-dominant and algebraic-dominant cases in \cite[\S9]{HR2019} or \cite{Guth2018}.\medskip

\begin{center}
 \begin{tabu}{|c|[1pt]c|c|} 
 \hline \xrowht[()]{10pt}
 Case & $\T[O_{j+1}]$ & $\widetilde{\T}[O_{j+1}]$ \\ \tabucline[1pt]{-} 
 \hline \xrowht[()]{10pt}
 Cellular-dominant & $\{(\theta,v) \in \T[\rho_j] : T_{\theta,v} \cap O_{j+1} \neq \emptyset \}$ & $\widetilde{\T}_{Z_{O_{j+1}}}$ \\ 
 \hline \xrowht[()]{10pt}
 Algebraic-dominant & $\T_{B,\mathrm{trans}}$ & $\widetilde{\T}_{Z_{O_{j+1}}}$\\ 
 \hline
\end{tabu}
\end{center}
\medskip
The set $\T_{B,\mathrm{trans}}$ is as defined in \cite[p.257]{HR2019} or \cite[p.129]{Guth2018}, but the precise choice of $\T[O_{j+1}]$ set is in fact unimportant for the purpose of establishing Property IV (the information is only included here as a reference to the arguments in \cite{Guth2018, HR2019}). The sets $\widetilde{\T}_{Z_{O_{j+1}}}$ are as defined in \eqref{rho tangent} for $\rho := \rho_{j+1}$ and $y = y_{O_{j+1}}$. 

\begin{remark} The additional decomposition according to the $\widetilde{\T}_{Z_{O_{j+1}}}$ wave packets is not carried out in the cellular-dominant case in either \cite{Guth2018} or \cite{HR2019} but is nevertheless useful in the argument: see \cite[p.254, fn 11]{HR2019} or \cite[Lemma 10.2]{GHI2019} for further details. 
\end{remark}

With the general setup above, $(\mathrm{IV})_{j+1}$ may be established as follows. By a combination of Lemma~\ref{repeat refine lem} and the basic orthogonality between the wave packets,
\begin{equation*}
    \|\tilde{f}_{O_{j+1}}|_{\W}\|_2^2 \lesssim  \|\tilde{g}_{O_{j+1}}|_{\widetilde{\T}[O_{j+1}] \cap \W^*}\|_2^2 \lesssim  \|\tilde{g}_{O_{j+1}}|_{\W^*}\|_2^2 
\end{equation*}
where $\W^*$ is the enlarged version of $\W$ defined in \S\ref{repeat refine sec}. Strictly speaking, these bounds should include additional rapidly decreasing error terms but, for simplicity, here and below these minor contributions are omitted. The transverse equidistribution estimate from Lemma~\ref{trans eq lem} implies that
\begin{equation}\label{IV proof 1}
    \|\tilde{g}_{O_{j+1}}|_{\W^*}\|_2^2 \lesssim \rho_j^{O(\delta_m)} \Big(\frac{\rho_j}{\rho_{j+1}}\Big)^{-m/2} \|g_{O_{j+1}}|_{\uparrow\W^*}\|_2^2
\end{equation}
whilst a second application of Lemma~\ref{repeat refine lem} and orthogonality yields
\begin{equation*}
    \|g_{O_{j+1}}|_{\uparrow\W^*}\|_2^2  \lesssim  \|f_{O_j}|_{\T[O_{j+1}] \cap (\uparrow\W^*)^*}\|_2^2 \lesssim  \|f_{O_j}|_{(\uparrow\W^*)^*}\|_2^2.
\end{equation*}
The set $(\uparrow\W^*)^*$ agrees with a set that is almost identical to $\uparrow\W$ except that certain constants in the definition of $\uparrow\W$ are slightly enlarged. By redefining $\uparrow\W$ using these larger constants, one concludes that
\begin{equation}\label{IV proof 2}
    \|\tilde{f}_{O_{j+1}}|_{\W}\|_2^2 \lesssim \rho_j^{O(\delta_m)} \Big(\frac{\rho_j}{\rho_{j+1}}\Big)^{-m/2} \|f_{O_j}|_{\uparrow\W}\|_2^2 .
\end{equation}

The inequality \eqref{IV proof 2} is in fact only useful in the algebraic-dominant case. The estimate is true in the cellular case but is not efficient since here the ratio of the scales $\rho_j/\rho_{j+1}$ is small compared to the $\rho^{O(\delta_m)}$ factor. Instead, one may replace \eqref{IV proof 1} in the above argument with the more elementary bound 
\begin{equation*}
    \|\tilde{g}_{O_{j+1}}|_{\W^*}\|_2^2 \lesssim  \|g_{O_{j+1}}|_{\uparrow\W^*}\|_2^2,
\end{equation*}
which follows directly from \eqref{basic W L2}. Arguing in this way, one may strengthen \eqref{IV proof 2} in the cellular case to an estimate without any additional $\rho_j^{O(\delta_m)}$. 

Henceforth assume the algebraic-dominant case holds; the cellular-dominant case follows almost identically using the refined version of \eqref{IV proof 2} described in the previous paragraph. Apply $(\mathrm{IV})_j$ to the right-hand side of \eqref{IV proof 2} to deduce that
    \begin{equation}\label{IV proof 3}
    \|\tilde{f}_{O_{j+1}}|_{\W}\|_2^2 \lesssim \rho_j^{O(\delta_m)} C_{j,\delta}^{\mathrm{IV}}(d, r) \Big(\frac{r}{\rho_{j+1}}\Big)^{-m/2} \|f_{O_j}|_{\uparrow^j(\uparrow\W)}\|_2^2 .
\end{equation}
It is claimed that $\uparrow^j(\uparrow\W) \subseteq \uparrow^{j+1}\W$. Once this is established, combing \eqref{IV proof 3} with basic $L^2$-orthogonality gives 
 \begin{equation*}
    \|\tilde{f}_{O_{j+1}}|_{\W}\|_2^2 \leq C(\deg \bZ, \delta) \rho_j^{\bar{C}\delta_m} C_{j,\delta}^{\mathrm{IV}}(d, r) \Big(\frac{r}{\rho_{j+1}}\Big)^{-m/2} \|f_{O_j}|_{\uparrow^j(\uparrow\W)}\|_2^2 .
\end{equation*} 
for suitable constants $C(\deg \bZ, \delta)$, $\bar{C} \geq 1$. Finally, from the formula \eqref{C IV formula} and the assumption that the algebraic-dominant case holds,
\begin{equation*}
   C(\deg \bZ, \delta) \rho_j^{\bar{C}\delta_m} C_{j,\delta}^{\mathrm{IV}}(d, r) \leq C_{j+1,\delta}^{\mathrm{IV}}(d, r),
\end{equation*}
provided the parameter $d$ is chosen to be sufficiently large, depending only on the admissible parameters and $\deg \bZ$. To see this, recall from the description of \texttt{[alg 1]} from \cite[\S9]{HR2019} that $\#_{\bta}(j+1) = \#_{\bta}(j)+1$ if the algebraic-dominant case holds.  This concludes the proof of $(\mathrm{IV})_j$, except for establishing the inclusion $\uparrow^j(\uparrow\W) \subseteq \uparrow^{j+1}\W$.

Let $(\theta,v) \in \uparrow^j(\uparrow\W)$ so that there exists some $(\theta_j, v_j) \in \uparrow\W$ such that
\begin{equation*}
    \dist(\theta,\theta_j) \leq c_j \rho_j^{-1/2} \quad \textrm{and} \quad  \dist(T_{\theta_j,v_j}(y_{O_j}),\, T_{\theta,v}(y) \cap B_{O_j}\big) \leq c_j r^{1/2+\delta}.
\end{equation*}
On the other hand, since $(\theta_j, v_j) \in \uparrow\W$ there exists some $(\theta_{j+1}, v_{j+1}) \in \W$ such that
\begin{gather*}
    \dist(\theta_j,\theta_{j+1}) \leq C \rho_{j+1}^{-1/2}; \\
    \dist(T_{\theta_{j+1},v_{j+1}}(y_{O_{j+1}}),\, T_{\theta_j,v_j}(y_{O_j}) \cap B_{O_{j+1}}\big) \leq C \rho_j^{1/2+\delta},
\end{gather*}
for an appropriate choice of $C$. At this point, fix the values of $c_j$ to be
\begin{equation*}
    c_j :=  C\sum_{i=0}^{j-1} 2^{-i/2} \leq C_{\circ} := C\sum_{i=0}^{\infty} 2^{-i/2} .
\end{equation*}
Since $\rho_{i+1} \leq \rho_i/2$ for all $0 \leq i \leq j$, it follows from the preceding displays that
\begin{gather*}
    \dist(\theta,\theta_{j+1}) \leq \big(c_j (\rho_{j+1}/\rho_j)^{1/2} + C\big) \rho_{j+1}^{-1/2} \leq c_{j+1}\rho_{j+1}^{-1/2}; \\
   \dist(T_{\theta_{j+1},v_{j+1}}(y_{O_{j+1}}),\, T_{\theta,v}(y) \cap B_{O_{j+1}}\big) \leq \big(c_j + C2^{-j/2}\big) r^{1/2+\delta} = c_{j+1} r^{1/2+\delta}.
\end{gather*}
Thus, $(\theta,v) \in \uparrow^{j+1}\W$, as required. \hfill $\Box$




\subsection{The modified stopping condition} The condition \texttt{[tang$^*$]} in \texttt{[alg~1$^*$]} is slightly different from the corresponding condition \texttt{[tang]} appearing in \texttt{[alg 1]} and, in particular, Condition~IV must hold in order to trigger \texttt{[tang$^*$]}. To incorporate this extra condition, the algorithm described in \cite[\S9]{HR2019} requires Conditions II, III and IV to hold for certain functions $f_{B, \mathrm{tang}}$. These functions are of the form $f_S := f_{O_j}|_{\T[S]}$ for some $\T[S] \subseteq \T[\rho_j]$, similar to those encountered in the previous subsection. This time $\T[S] := \T_{B,\mathrm{tang}}$, where the latter set is as defined in \cite[p.237]{HR2019} or \cite[p.129]{Guth2018}. As before, the exact form of the set $\T[S]$ is not important for the purposes of verifying the Condition IV. Thus, one wishes to show that
\begin{equation}\label{stopping 1}
    \|\tilde{f}_{S} |_{\W} \big\|_2^2 \lesssim \|f_{O_j} |_{\uparrow \W} \big\|_2^2.
\end{equation}
This inequality is easily deduced using the arguments of the previous subsection. In particular, \eqref{basic W L2} and Lemma~\ref{repeat refine lem} together imply that
\begin{equation*}
    \big\|\big(f_{S}\big)\;\widetilde{}\; |_{\W} \big\|_2^2 \lesssim \|f_S |_{\uparrow \W}\|_2^2 \lesssim \|f_{O_j} |_{\T[S] \cap (\uparrow \W)^*}\|_2^2.
\end{equation*}
The desired estimate \eqref{stopping 1} now follows from the $L^2$-orthogonality between the wave packets, provided, as before, that the constants in the definition of $\uparrow \W$ are slightly enlarged.  



\subsection{The second algorithm: a sketch} The multigrain decomposition from \S\ref{multigrain dec sec} is obtained by repeatedly applying \texttt{[alg~1$^*$]} as part of a recursive procedure described in \texttt{[alg 2]} in \cite[\S10]{HR2019}. Here a brief sketch of this process is given. 

At stage $0$, one begins with the input of the multigrain decomposition from \S\ref{multigrain dec sec}. After the $\ell$th stage a family of functions $f_{\vec{S}_{\ell}}$ has been constructed, indexed by $\vec{S}_{\ell} \in \vec{\Sc}_{\ell}$ where $\vec{\Sc}_{\ell}$ is a collection of level $\ell$ multigrains at some scale $r_{\ell}$ and of complexity $O_{\varepsilon}(1)$. Each function $f_{\vec{S}_{\ell}}$ is tangent to $S_{\ell}$, the codimension $\ell$ grain forming the final component of $\vec{S}_{\ell}$. Furthermore, the functions satisfy suitable ``level $\ell$ variants'' of the properties \eqref{P 1} to \eqref{P 4}: see \cite[\S10]{HR2019} for details.

To pass to the next stage of the construction, apply \texttt{[alg~1$^*$]} to each function $f_{\vec{S}_{\ell}}$. Notice that these functions satisfy the tangency conditions required in the input of the algorithm. In each case, either \texttt{[alg~1$^*$]} terminates due to \texttt{[tiny]} or due to \texttt{[tang$^*$]}. Suppose that the inequality
  \begin{equation}\label{tiny case}
      \sum_{\vec{S}_{\ell} \in \vec{\Sc}_{\ell}} \|Ef_{\vec{S}_{\ell}}\|_{\BL{k,A_{\ell}}^{p_{\ell}}(B_{r_{\ell}})}^{p_{\ell}} \leq 2  \sum_{\vec{S}_{\ell} \in \vec{\Sc}_{\ell,\textrm{\texttt{tiny}}}} \|Ef_{\vec{S}_\ell}\|_{\BL{k,A_{\ell}}^{p_{\ell}}(B_{r_{\ell}})}^{p_{\ell}}
\end{equation}
holds, where the right-hand summation is restricted to those $S_{\ell} \in \vec{\Sc}_{\ell}$ for which \texttt{[alg 1$^*$]} terminates owing to the stopping condition \texttt{[tiny]}. In this case, the process terminates and $m := \ell$. Defining the functions $f_O$ appropriately via  \texttt{[alg 1$^*$]}, one may verify the properties of the multigrain decomposition from \S\ref{multigrain dec sec}: see \cite[\S\S9-10]{HR2019} and the following discussion for further details. 

Alternatively, if \eqref{tiny case} fails, then necessarily
  \begin{equation*}
      \sum_{\vec{S}_{\ell} \in \vec{\Sc}_{\ell}} \|Ef_{\vec{S}_{\ell}}\|_{\BL{k,A_{\ell}}^{p_{\ell}}(B_{r_{\ell}})}^{p_{\ell}} \leq 2 \sum_{\vec{S}_{\ell} \in \vec{\Sc}_{\ell,\textrm{\texttt{tang}}}} \|Ef_{\vec{S}_{\ell}}\|_{\BL{k,A_{\ell}}^{p_{\ell}}(B_{r_{\ell}})}^{p_{\ell}},
\end{equation*}
where the right-hand summation is restricted to those $S_{\ell} \in \vec{\Sc}_{\ell}$ for which \texttt{[alg 1$^*$]} does not terminate owing to \texttt{[tiny]} and therefore terminates owing to \texttt{[tang$^*$]}. For each $\vec{S}_{\ell} \in \vec{\Sc}_{\ell}$ there exists a collection $\Sc_{\ell+1}[\vec{S}_{\ell}]$ of codimension $\ell+1$ grains of degree $O_{\varepsilon}(1)$ which arise from \texttt{[tang$^*$]} and, in particular, satisfy the Conditions I to IV with $f$ replaced with $f_{\vec{S}_{\ell}}$. By appropriately pigeonholing, one may further assume that all the grains in  $\Sc_{\ell+1}[\vec{S}_{\ell}]$ have a common scale $r_{\ell+1} < r_{\ell}$. A family of level $\ell+1$ multigrains is then defined by
\begin{equation*}
    \vec{\Sc}_{\ell+1} := \big\{ (\vec{S}_{\ell}, S_{\ell+1}) : S_{\ell+1} \in \Sc_{\ell+1}[\vec{S}_{\ell}] \big\},
\end{equation*}
whilst $f_{\vec{S}_{\ell+1}} := (f_{\vec{S}_{\ell}})_{S_{\ell+1}}$, where the right-hand function satisfies the properties stated in \texttt{[tang$^*$]}.




\subsection{Ensuring Property iv)} The procedure sketched in the previous subsection is precisely \texttt{[alg 2]} from \cite{HR2019} (which in turn corresponds to the induction-on-dimension used in \cite{Guth2018}). The only modification required for the purposes of this article is to construct the functions $f_{\vec{S}_{\ell}}^{\#}$ described in \S\ref{multigrain dec sec} and ensure that Property iv) from \S\ref{multigrain dec sec} holds. This is achieved using the additional Property IV of \texttt{[alg 1$^*$]}. 

For $\vec{S}_{\ell} \in \vec{\Sc}_{\ell}$ and $0 \leq j \leq \ell$, if $\big(S_j, B(y_j,r_j)\big)$ denotes the codimension $j$ component of $\vec{S}_{\ell}$, then let $\T_{\mathrm{tang}}[S_j]$ denote the set of all scale $r_j$ wave packets which are $r_j^{-1/2+\delta_j}$-tangent to $S_j$ in $B(y_j, r_j)$. Thus, if $\vec{S}_{\ell} \preceq \vec{S}_j$ for some $0 \leq j \leq \ell$, then the function $f_{\vec{S}_j}$ is concentrated on wave packets belonging to $\T_{\mathrm{tang}}[S_j]$.

For $1 \leq \ell \leq m$, given $\W \subseteq \T[r_{\ell}]$, let $\uparrow\uparrow_{\ell} \!\!\W$ denote the set of wave packets $(\theta, v) \in \T[r_{\ell-1}]$ for which there exists some $(\tilde{\theta}, \tilde{v}) \in \W$ satisfying
\begin{equation*}
    \dist(\tilde{\theta},\theta) \leq C_{\circ} r_{\ell}^{-1/2} \quad \textrm{and} \quad \dist\big(T_{\tilde{\theta},\tilde{v}}(y_{\ell}),\, T_{\theta,v}(y_{\ell-1}) \cap B(y_{\ell},r_{\ell})\big) \leq C_{\circ} r_{\ell-1}^{-1/2+\delta}.
\end{equation*}
Property IV of \texttt{[alg 1$^*$]} and the definition of the stopping condition \texttt{[tang$^*$]} together imply that for $1 \leq \ell \leq m$ and $\W \subseteq \T[r_{\ell}]$, the inequality 
\begin{equation}\label{alg 2 1} \big\|f_{\vec{S}_{\ell}}|_{\W}\big\|_2^2 \lesssim_{\varepsilon} \Big(\frac{r_{\ell-1}}{r_{\ell}}\Big)^{-(\ell-1)/2}D_{\ell}^{\delta}   R^{O(\varepsilon_{\circ})}\big\|f_{\vec{S}_{\ell-1}}|_{\uparrow\uparrow_{\ell} \W}\big\|_2^2 
\end{equation}
holds whenever $\vec{S}_{\ell-1} \in \vec{\Sc}_{\ell-1}$, $\vec{S}_{\ell} \in \vec{\Sc}_{\ell}$ and $\vec{S}_{\ell} \preceq \vec{S}_{\ell-1}$.

 Construct a sequence of sets $\W_j \subseteq \T_{\mathrm{tang}}[S_j]$ for $0 \leq j \leq \ell$ recursively as follows:
\begin{itemize}
    \item Set $\W_{\ell} := \T_{\mathrm{tang}}[S_{\ell}]$.
    \item Assuming $\W_t, \dots, \W_{\ell}$ have already been constructed for some $1 \leq t \leq \ell$, define
    \begin{equation*}
        \W_{t-1} := \T_{\mathrm{tang}}[S_{t-1}] \cap \big(\!\uparrow\uparrow_t\! \! \W_t\big)^*.
    \end{equation*}
\end{itemize}
For each $0 \leq t \leq \ell$, the tubes belonging to $\W_t$ satisfy a version of the hypothesis from Definition~\ref{nested tube condition}. In particular, each $(\theta_t,v_t) \in \W_t$ satisfies the following:\medskip

\noindent \textbf{Nested tube hypothesis.} There exist $(\theta_i,v_i) \in \T[r_i]$ for $t +1  \leq i \leq \ell$ such that
    \begin{enumerate}[i)]
    \item $\dist(\theta_i,\theta_j) \lesssim r_j^{-1/2}$,
    \item $\dist\big(T_{\theta_j,v_j}(y_j), \, T_{\theta_i,v_i}(y_i) \cap B(y_j,r_j)\big) \lesssim r_i^{1/2+\delta}$,
    \item $T_{\theta_j,v_j}(y_j) \subset N_{r_j^{1/2+\delta_j}}S_j$
\end{enumerate}
hold for all $t \leq i \leq j \leq \ell$.\medskip

Furthermore, for each $0 \leq t \leq \ell$ the inequality 
\begin{equation}\label{alg 2 2}
    \|f_{\vec{S}_{\ell}}\|_2^2 \lesssim \prod_{j=t+1}^{\ell} \Big(\frac{r_{j-1}}{r_j}\Big)^{-\frac{j-1}{2}}D_j^{\delta}  \|f_{\vec{S}_t} |_{\W_t}\|_2^2
\end{equation}
holds up to the inclusion of a rapidly decaying error term. Indeed, by \eqref{alg 2 1} above,
\begin{equation*}
    \|f_{\vec{S}_j} |_{\W_j}\|_2^2 \lesssim_{\varepsilon} \Big(\frac{r_{j-1}}{r_j}\Big)^{-\frac{j-1}{2}}D_j^{\delta}   R^{O(\varepsilon_{\circ})}\big\|f_{\vec{S}_{j-1}}|_{\uparrow_j\W_j}\big\|_2^2.
\end{equation*}
Since $f_{\vec{S}_{j-1}}$ is concentrated on wave packets belonging to $\T_{\mathrm{tang}}[S_{j-1}]$, one has
\begin{equation*}
    f_{\vec{S}_{j-1}} = f_{\vec{S}_{j-1}}|_{\T_{\mathrm{tang}}[S_{j-1}]} + \mathrm{RapDec}(r)\|f\|_2.
\end{equation*}
Combining the two preceding displays together with Lemma~\ref{repeat refine lem}, one deduces that
\begin{equation*}
        \|f_{\vec{S}_j} |_{\W_j}\|_2^2 \lesssim_{\varepsilon} \Big(\frac{r_{j-1}}{r_j}\Big)^{-\frac{j-1}{2}}D_j^{\delta}R^{O(\varepsilon_{\circ})}   \big\|f_{\vec{S}_{j-1}}|_{\W_{j-1}}\big\|_2^2
\end{equation*}
and this inequality may be applied recursively to deduce \eqref{alg 2 2}.

To conclude the proof, simply define $f_{\vec{S}_{\ell}}^{\#} := f|_{\W_0}$, noting that the desired properties then immediately follow from the preceding discussion.\hfill $\Box$




\appendix

\section{Deriving the asymptotic for the linear exponents}\label{appendix}

Here the derivation of the $\lambda$ coefficient featured in Theorem \ref{asymptotic thm} is described. The first author thanks Keith M. Rogers for the following argument. Begin by noting that
\begin{equation}\label{bounds}
\frac{2i+1}{2(i+1)+1}\ge\frac{2i}{2i+1}\frac{2(i+1)}{2(i+1)+1}\ge \frac{i}{i+1},
\end{equation}
so that, by using the lower bound and telescoping, 
$$
\Big(\prod_{i=k}^{n-1}\frac{2i}{2i+1}\Big)^2=\frac{2k}{2k+1}\frac{2n+1}{2n}\prod_{i=k}^{n-1}\frac{2i}{2i+1}\frac{2(i+1)}{2(i+1)+1}\ge\frac{2n+1}{2k+1}\frac{k^2}{n^2}.
$$
Taking the square root and plugging this into the definition of $p_n(k)$, one obtains 
\begin{align*}
p_n(k)\le2+\frac{6}{2(n-1)+(k-1)(\frac{2n+1}{2k+1})^{1/2}\frac{k}{n}}.
\end{align*}
The analogous argument, using the upper bound from \eqref{bounds}, yields
\begin{align*}
p_n(k)\ge2+\frac{6}{2(n-1)+(k-1)(\frac{k}{n})^{1/2}}.
\end{align*}
Taking $k=\nu n + O(1)$ for some $0<\nu<1$, it follows that, asymptotically,
\begin{equation}\label{boundbel0}
p_n(k)=2+\frac{6}{2+\nu^{3/2}} n^{-1}+O(n^{-2}).
\end{equation}
On the other hand, for $k=\nu n + O(1)$, the constraint
\begin{equation*}
    p \geq 2+\frac{4}{2n-k}
\end{equation*}
coming from the Bourgain--Guth argument (c.f. \eqref{BG constraints}) can be rewritten as
\begin{equation}\label{boundbel}
p\ge 2+\frac{4}{2-\nu}n^{-1} +O(n^{-2}).
\end{equation}
Optimal choice of $\nu$ corresponds to the value at which linear coefficients in \eqref{boundbel0} and \eqref{boundbel} are equal. This occurs when $\nu^{1/2}$ solves  the cubic equation
\begin{equation*}
    2x^{3}+3x^2-2=0;
\end{equation*}
the derivation of this condition is presented in the appendix. Cardano's formula shows that the unique real root of this equation is given by the irrational number\footnote{One can immediately see that the root must be irrational by applying Eisenstein's criterion to the shifted polynomial $2(x+1)^3 + 3(x+1)^2 - 2$.}
\begin{equation*}
    \nu^{1/2}=\Big(\frac{3}{8} + \frac{1}{8^{1/2}}\Big)^{1/3}+\Big(\frac{3}{8} - \frac{1}{8^{1/2}}\Big)^{1/3}-\frac{1}{2} = 0.67765... .
\end{equation*}
Plugging this back into \eqref{BG constraints} yields  \eqref{extension} in the range
\begin{equation*}
    p>2+\lambda n^{-1}+O(n^{-2})
\end{equation*}
with $\lambda = \frac{4}{2-\nu}= 2.59607...$.




\bibliography{Reference}
\bibliographystyle{amsplain}

\end{document}